\documentclass[reqno, 11pt, a4paper]{amsart}
\usepackage{amsmath, amssymb, amsthm}  
\usepackage{comment}
\usepackage{graphicx, color}
\usepackage{mathtools}
\usepackage{mathrsfs, enumerate}
\usepackage{hyperref}
\usepackage[truedimen,top=3truecm, bottom=3truecm, left=2.5truecm, right=2.5truecm, includefoot]{geometry} 
\usepackage{cancel}

\usepackage[capitalise]{cleveref}

\usepackage{tikz}
\usetikzlibrary{shapes,arrows,positioning,fit,backgrounds,calc,patterns}
\usepackage{pgfplots}
\pgfplotsset{compat=1.15}

\usepackage[active]{srcltx}
\usepackage{xcolor}
\definecolor{columbiablue}{rgb}{0.61, 0.87, 1.0}
\usepackage{wrapfig}
\usepackage[]{pstricks}
\usepackage{fancyhdr}
\definecolor{sandstone}{HTML}{786D5F}
\definecolor{beaublue}{rgb}{0.74, 0.83, 0.9}
\definecolor{cherryblossompink}{rgb}{1.0, 0.72, 0.77}
\definecolor{light-gray}{gray}{0.95}
\definecolor{apricot}{rgb}{0.98, 0.81, 0.69}

\usepackage[percent]{overpic}

\newcommand{\kh}[1]{{\color{blue} #1}}
\newcommand{\so}[1]{{\color{red} #1}}

\newcommand{\ve}{\varepsilon}

\usepackage[normalem]{ulem}
\newcommand{\Add}[1]{\textcolor{black}{#1}}	
\newcommand{\Erase}[1]{\if0{#1}\fi}

\begin{document}

\title 
[Nonlinear fluctuations for weakly anharmonic stochastic oscillators]
{Nonlinear fluctuations for a chain of weakly anharmonic oscillators with stochastic perturbation}

\author[Kohei Hayashi]{Kohei Hayashi}
\address{Kohei Hayashi, Department of Mathematics, Graduate School of Science, Osaka University 
\\
\emph{and} RIKEN Center for Interdisciplinary Theoretical and Mathematical Sciences} 
\email{khayashi@math.sci.osaka-u.ac.jp}

\author{Stefano Olla}
\address{Stefano Olla, CEREMADE,
Universit\'e Paris Dauphine - PSL Research University \\
\emph{and}  Institut Universitaire de France\\
\emph{and} GSSI, L'Aquila}
\email{olla@ceremade.dauphine.fr}

\begin{abstract}
We study the fluctuations of the phonon modes in a one-dimensional chain of anharmonic oscillators where the deterministic Hamiltonian dynamics is perturbed
by random exchanges of momentum between nearest neighbor particles.
There are three locally conserved quantities: volume, momentum and energy.
We study the evolution in equilibrium of the fluctuation fields of the two phonon modes (linear combination of the volume stretch and momentum), on a diffusive space-time scale after recentering on their sound velocities.
We show that, weakening the anharmonicity with the scale parameter, the recentered phonon fluctuations fields
converge to the stationary solutions of two uncoupled stochastic Burgers equations. 
The nonlinearity in the Burgers equation depends on the presence of a cubic term in the anharmonic potential
(corresponding to the $\alpha$-FPUT dynamics).
Main ingredients of the proof, based on a compactness argument for the Dynkin's martingale decomposition, are the second-order Boltzmann-Gibbs principle, as well as equipartition of energy, to characterize the nonlinear term and Riemann-Lebesgue estimates showing that fields with diverging velocity to different directions have no interaction in the limit. 
\end{abstract}

\keywords{KPZ Equation, Stochastic Burgers Equation, Interacting Oscillators}
\subjclass[2020]{60H15, 60K35, 82B44}

\maketitle

\theoremstyle{plain}
\newtheorem{theorem}{Theorem}[section] 
\newtheorem{lemma}[theorem]{Lemma}
\newtheorem{corollary}[theorem]{Corollary}
\newtheorem{proposition}[theorem]{Proposition}
\newtheorem{conjecture}[theorem]{Conjecture}

\theoremstyle{definition}
\newtheorem{definition}[theorem]{Definition}
\newtheorem{remark}[theorem]{Remark}
\newtheorem{assumption}[theorem]{Assumption}
\newtheorem{example}[theorem]{Example}

\makeatletter
\renewcommand{\theequation}{%
\thesection.\arabic{equation}}
\@addtoreset{equation}{section}
\makeatother

\makeatletter
\renewcommand{\p@enumi}{A}
\makeatother

\newcommand{\ven}{\varepsilon_n}

\section{Introduction}
The theoretical and numerical investigation of the statistical mechanics of one-dimensional chain
of oscillators attracted large attention since the proposal of Fermi-Pasta-Ulam-Tsingou
of these dynamics for testing the approach to thermal equilibrium \cite{fermi1955studies}.
Even though this issue is not yet closed (and mathematically totally open), it is believed
that for generic non-harmonicity in the interaction and at energy density enough high and the
system very large, it thermalize in relatively short times.
These means that in such conditions only three conserved quantities are present:
the volume stretch, the momentum and the energy (see later the precise definition).
We refer to this property as the \emph{ergodicity of the infinite system}:
mathematically this is expressed by the statement that \cite{fritz1994stationary}
\emph{the only stationary and translation invariant probability measures
locally regular are given by the
Gibbs measures parametrized by temperature, momentum and tension.}
We refer these measures as the equilibrium measures.
Important counterexamples are given by completely integrable systems like harmonic
oscillators and the Toda lattice, where the number of conserved quantities is equal to the
number of degree of freedom.

The conservation of momentum implies that, under the ergodicity assumption mentioned above,
there is a macroscopic ballistic transport of the three conserved quantities, governed by the
compressible Euler equations. Then starting the system in one given equilibrium measure,
the initial fluctuations of the three conserved quantities evolve deterministically following
the characteristics of the linearized Euler equations. We can describe this evolution in
a simple way in the case the average momentum is null. We refer at the energy
fluctuation as the \emph{energy or thermal mode} and the two linear combinations of the
volume stretch and momentum at the \emph{phonon or sound modes}
(see \eqref{eq:xi_newdef} for the definition).
Then under hyperbolic (Euler) rescaling of space and time
the phonon modes evolve rigidly with opposite equal velocity,
given by the \emph{sound velocity}, while the energy mode does not move.

The question is then about the space-time
scales, larger that the Euler scaling, where we see a random evolution of the energy and phonon modes
and what is the nature of these evolutions.

Numerical evidence proved since the 90s that energy mode evolves in a superdiffusive
time scale \cite{lepri1997heat}. The seminal work of Van Beijeren \cite{vanbeij2012}
and Spohn \cite{spohn2014nonlinear} proposed the
\emph{nonlinear fluctuating hydrodynamics theory} (NLFH)
as a mesoscopic approach to catch
the superdiffusive macroscopic broadening of the modes.
This consist in developing the Euler equations up to second order and adding a dissipative
randomness given by gradient of space-time white noises.
These equations constitute a system of \emph{stochastic Burgers equations}.
Then a mode coupling analysis \cite{spohn2014nonlinear}
connects the asymptotic behavior of the
correlation functions of the phonon modes obtained the
the solutions of these nonlinear fluctuation equations  to the KPZ universal scaling function,
while the correlations of the energy mode converges to the one of
the 5/3-L\'evy distribution. Except in the case of even potential and 0 tension,
where the NLFH predicts a diffusive asymptotic phonon broadening, whereas the energy mode broadening is given by the 3/2-L\'evy distribution.

Rigorous mathematical results have been obtained by adding noise directly on the microscopic Hamiltonian dynamics.
Noise is such that it conserves 
the three modes, destroying all other conserved quantities, giving the required
ergodicity to the infinite system \cite{fritz1994stationary}. 
For example an exchange of momentum between nearest-neighbor particles in the chain at random times (like a random elastic collision), but sometimes, for technical reasons, also randomness involving the positions.
Under hyperbolic scaling, Euler equations in the smooth regime are obtained under some conditions on the potential (\cite{olla1993hydrodynamical}, and more specific for the anharmonic chain in \cite{braxmeier2014hydrodynamic}),
while linear fluctuations are proven in \cite{olla2020equilibrium}. 
In the Euler scaling macroscopic equations are not affected by the microscopic noise.

Beyond Euler scaling it is very hard to obtain mathematical results, even in presence
of such conservative noise. Results can be obtained for the Harmonic chain where
explicit calculations can be performed thanks to Fourier analysis.
The local random exchange of momentum destroys the integrability of the harmonic chain,
so that the three modes can be studied on larger time scales.
In this case, in the Euler scale, linear wave equation governs the evolution of the phonon modes, while
energy mode does not move if average momentum is null.
After re-centering the phonon modes on their sound velocity, they evolve diffusively
\cite{komorowski2016ballistic}, while the energy mode evolves on a superdiffusive scale
governed by a 3/2-L\'evy distribution \cite{jara2015superdiffusion}.
Interestingly, this is the same universality class predicted by the NLFH for even potential and 0 tension \cite{spohn2014nonlinear}. 

In this article, we study the effect of a weak anharmonicity on the asymptotic behavior of
the phonon modes fluctuations. We consider an anharmonic interaction potential
$V(q_i - q_{i-1})$ where $V$ is a smooth nonlinear function such that $V(0) = V'(0) = 0$.
Then we scale it as $\ve^{-2} V(\ve r)$ and we consider the behavior as $\ve\to 0$. 
A Taylor expansion gives
\begin{equation}
\label{eq:taylor-1}
\ve^{-2} V(\ve r) = \frac{c_2}{2!}r^2 
+ \frac{c_3}{3!} \varepsilon r^3
+ \frac{c_4}{4!} \varepsilon^2 r^4 
+ O(\varepsilon^3),  \qquad c_k=V^{(k)}(0).
\end{equation}
Since in equilibrium at temperature $T = \beta^{-1}$ the variance of $q_i - q_{i-1}$
is proportional to $T$, this rescaling corresponds to study the first-order term in the expansion
of the dynamics in the low temperature limit (i.e. $T=\ve^2$).

Let $n\to\infty$ the scaling parameter, i.e., the typical macroscopic distance between particles is $n^{-1}$. Then we choose $\ve = \ve_n \to 0$ as $n\to \infty$, i.e. we rescale the anharmonicity
together with the space-time scaling.

The main result in this article is that, with the choice $\ve_n = n^{-1/2}$, the fluctuation fields of the two phonon modes converge, under diffusive rescaling of space-time,
to the \emph{stationary energy solutions} (see \cref{sec-ssbe})
$u^+, u^-$ of the stochastic Burgers equations
\begin{equation}\label{eq:drift-sbe-intro}
\partial_t u^{\pm}
=\frac{\gamma}{4} \partial_x^2 u^\pm
\pm \frac{ c_3}{8c_2^2}\partial_x (u^\pm)^2 
\pm D_V \partial_x u^\pm
+ \sqrt{\gamma\beta^{-1}}
\partial_x \dot{W}^\pm,
\end{equation}
where 
\begin{equation}
\label{eq:moving_frame_modification_constant} 
D_V = \Add{\frac{c_2c_4-c_3^2}{4c_2^{5/2}\beta} } ,
\end{equation}
$\gamma>0$ is a parameter of the intensity of the microscopic random exchanges,
$\beta^{-1}$ is the temperature of the equilibrium distribution and $\dot{W}^+(t,x), \dot{W}^-(t,x)$ are two independent standard space-time white noises.

There are two important remarks about this result:
\begin{itemize}
\item The presence of the cubic term in the interaction ($c_3 \neq 0$) is responsible for the nonlinear term appearing in \eqref{eq:drift-sbe-intro}.
\item In absence of
cubic and quartic anharmonicity ($c_3 = c_4 =0$) we have the same diffusive behavior as
in the harmonic chain proved in \cite{komorowski2016ballistic}.
\end{itemize}

About the energy mode fluctuations, we prove that with the choice $\ve_n = n^{-1/2}$
the contribution given by the nonlinearity is negligible in the superdiffusive scaling,
so that the macroscopic behavior is the same as in the harmonic case, i.e. 3/2-L\'evy
as proven in \cite{jara2015superdiffusion}.

With stronger choice of the anharmonicity, i.e. for $\ve_n = n^{-\mathfrak b}$ with $\mathfrak b\in (1/4,1/2)$,
our result still holds for shorter time scales, but keeping the random exchange rate unchanged
(i.e. typically $n^2$ random exchanges per unit time), see Remark \ref{rem-b}.

After the seminal work of Bertini-Giacomin \cite{bertini1997stochastic}, stochastic Burgers
equations have been derived from various types of weakly asymmetric stochastic dynamics
\cite{gonccalves2014nonlinear, gonccalves2015stochastic, jara2019scaling, ahmed2022microscopic, gonccalves2023derivation, gonccalves2024characterization, bernardin2021derivation, butelmann2021scaling}.  
While in \cite{bertini1997stochastic} the Cole-Hopf mapping into the stochastic heat equation was the main tool, in \cite{gonccalves2014nonlinear} the robust concept of \emph{energy solution} was introduced (at least in equilibrium) and used it in the other cited work.
Strong asymmetric dynamics (i.e. the intensity of the antisymmetric part of the interaction is equal to the one of the symmetric part) have been recently studied in
\cite{jara2020stationary, hayashi2023derivation, hayashi2024derivation}, where the scaling parameter is in the interaction.
The present article is inspired by \cite{hayashi2024derivation} as well as
\cite{gonccalves2023derivation} where it is analyzed the fluctuations of the two conserved quantities
in the anharmonic Bernardin-Stoltz chain \cite{ahmed2022microscopic, bernardin2018nonlinear}
(see also \cite{bernardin2012anomalous, bernardin2014anomalous, bernardin2018weakly, bernardin2018interpolation}).

Our proof is based on a standard compactness argument starting from the Dynkin's martingale decomposition, and after showing each term in the martingale decomposition is tight, we will show that any limiting point is characterized by the stationary energy solution of the SBE, i.e., a notion of solution as a martingale problem which is introduced in \cite{gonccalves2014nonlinear} and whose uniqueness is proved in \cite{gubinelli2018energy}. 
In addition, a quadratic field which converges to the nonlinear term of the SBE is dealt with the second-order Boltzmann-Gibbs principle, of which proof we do not need any spectral gap estimate. 
Moreover, we will show a kind of the Riemann-Lebesgue lemma, which enables us to neglect some fields with wrong, divergent frame. 
A novelty of the main result is the robustness of the strategy, noting that the expansion \eqref{eq:taylor-1} is valid for any nonlinear function $V$. 
Also notice that, unlike all previous results on SBE, the microscopic noise in the dynamics is highly
degenerate, as it acts only on the velocities. 

This paper is organized as follows. 
First, in Section \ref{sec:model}, we give a precise definition of the microscopic dynamics, and after recalling previous literature
we state the main result. 
In \cref{sec:sketch}, we give a sketch of the proof, beginning from the Dynkin's martingale formula, where the emergence of nonlinear, linear and viscosity terms of the limiting SBE is explained, and particularly the martingale decomposition is rewritten in an advantageous form.   
Next, in Section \ref{sec:energy-fluctuations}, we give a proof of equipartition of energy, which is essential to see the emergence of nonlinear (i.e., degree-two) term.  
In \cref{sec:2bg} we prove the second-order Boltzmann-Gibbs principle that generates the nonlinear term
from the degree-two terms present in the martingale decomposition. 
In \cref{sec:riemann_lebesgue_estimates}, we prove some Riemann-Lebesgue estimates whose implications is
that fields with wrong velocity do not contribute in the limit.  
For completion of the proof, in Section \ref{sec:tightness} we show that each term in the martingale decomposition is tight, and then characterize the limiting points in Section \ref{sec:identification_limit_points}. 
Finally, the notion of stationary energy solution of SBE and some auxiliary estimates are provided in Appendices.

\subsection*{Notation}
Given two real-valued functions $f$ and $g$ depending on a variable $u\in\mathbb R^d$, we write $f(u) \lesssim g(u)$ if there exists a constant $C>0$ such that $f(u) \le Cg(u)$ for any $u$. 
Moreover, we write $f=O(g)$ (resp. $f=o(g)$) in the neighborhood of $u_0$ if $|f|\le |g|$ in the neighborhood of $u_0$ (resp. $\lim_{u\to u_0}f(u)/g(u)=0$). 
Sometimes it will be convenient to make precise the dependence of the constant $C$ on some extra parameters and this will be done by the standard notation $C(\lambda)$ if $\lambda$ is the extra parameter. 
\Erase{Morever}\Add{Moreover}, we denote by $\langle \cdot,\cdot\rangle_{L^2(\mathbb R)}$ the inner product in \Erase{$L^(\mathbb R)$}\Add{$L^2(\mathbb R)$}, i.e., for any $f,g\in L^2(\mathbb R)$ 
\begin{equation*}
\langle f, g \rangle_{L^2(\mathbb R)}
\coloneqq \int_{\mathbb R} f(s) g(s) dx ,
\end{equation*}
and by $\| \cdot \|_{L^2(\mathbb R)}$ the $L^2(\mathbb R)$-norm, i.e., $\| f\|_{L^2(\mathbb R)}=\langle f,f\rangle_{L^2(\mathbb R)}^{1/2}$. 
Finally, we prepare some derivative and shift operators acting on discrete functions. 
For each real sequence $g=(g_j)_{j\in\mathbb Z}$, define 
\begin{equation}
\label{eq:definition_discrete_gradient}
\nabla^+ g_j = g_{j+1}-g_j ,\quad
\nabla^- g_j = g_{j-1}-g_j ,\quad
\Delta g_j = g_{j+1} + g_{j-1} - 2g_j 
\end{equation}
and for each $n>0$, define
\begin{equation}
\label{eq:definition_discrete_derivative}
\nabla^{n} g_j 
= n \nabla^+ g_j, \quad
\Delta^n g_j = n^2 \Delta g_j .
\end{equation}

\section{Model and Result}
\label{sec:model}
\subsection{Microscopic Dynamics} 
In what follows, we consider a chain of coupled anharmonic oscillators $\mathfrak p = (p_j)_{j\in\mathbb Z} , \mathfrak q = (q_j)_{j\in\mathbb Z} \in \mathbb{R}^{\mathbb{Z}}$ where $p_j$ and $q_j$ denote the position and momentum of an oscillator labeled by $j\in\mathbb Z$. 
We set $\mathfrak r=(r_j)_{j\in\mathbb Z}$ where 
\begin{equation*}
r_j = q_j - q_{j-1}. 
\end{equation*} 
In what follows, we consider microscopic dynamics of $(\mathfrak p, \mathfrak r)$. 
Throughout this paper, we fix a generic nonlinear function $V$ which satisfies the following condition.  

\begin{assumption}
\label{asm:potential}
Let $V:\mathbb R\to \mathbb R$ be smooth, non-negative 
function such that $V(0)=V'(0)=0$ and $V''(0) >0$.
Moreover, assume for each $k\in\{0,\ldots,5\}$ that, the derivative $V^{(k)}(r)$ has at most exponential growth, that is, there exists a constant $\eta_V>0$ such that 
\begin{equation*}
\max_{0\le k\le 5} \sup_{r\in\mathbb{R}} \big|  e^{-\eta_V |r|} V^{(k)}(r) \big| <+\infty.
\end{equation*}
Here we used the convention $V^{(0)}(\cdot)=V(\cdot)$. 
\end{assumption}

Let $n>0$ be a scale parameter and define
$V_n(\cdot) = \varepsilon_n^{-2} V(\varepsilon_n \cdot)$
with $\varepsilon_n \to0$ as $n\to\infty$.  
Then, by Taylor's theorem, we can expand 
\begin{equation*}
V_n(r) 
= \frac{c_2}{2!}r^2 
+ \frac{c_3}{3!} \varepsilon_n r^3
+ \frac{c_4}{4!} \varepsilon_n^2 r^4 
+ O(\varepsilon^3)
\end{equation*} 
where we set 
\begin{equation}
\label{eq:differential_coefficient_definition}
c_k=V^{(k)}(0).
\end{equation}
Thus, $V_n$ can be viewed as a weak perturbation of the quadratic function $r\mapsto r^2/2$.
In what follows, we consider a chain of oscillators driven by the weakly
anharmonic potential $V_n(\cdot)$, with a stochastic perturbation.  
Let $\gamma>0$ be a positive constant.  
For any local smooth function
$f:\mathbb R^\mathbb Z\times \mathbb R^\mathbb Z \to\mathbb{R}$, define \begin{equation*}
S f (\mathfrak r,\mathfrak p) 
= 
\frac{\gamma}{2} \sum_{j\in\mathbb{Z}}
\big( f(\mathfrak r,\mathfrak p^{j,j+1}) - f(\mathfrak r,\mathfrak p) \big)
\end{equation*}
where 
$\mathfrak p^{j,j+1}$ is the configuration obtained after exchanging 
$\mathfrak p_j$ and $\mathfrak p_{j+1}$.
Let $A$ be the generator corresponding to the deterministic Hamiltonian dynamics: 
\begin{equation*}
\begin{aligned}
A
= \sum_{j\in\mathbb Z} 
\big( p_j \partial_{q_j}
+ \nabla^+ V_n'(r_j) \partial_{p_j} \big)
=  \sum_{j\in\mathbb Z} 
\big(-\nabla^- p_j \partial_{r_j}
+ \nabla^+ V_n'(r_j) \partial_{p_j} \big)
\end{aligned}
\end{equation*}
where recall that $\nabla^\pm$ are defined in \eqref{eq:definition_discrete_gradient}. 
In the whole article, we assume that the potential $V$ is such that the dynamics, which is generated by 
\begin{equation}
\label{eq:gen}
L = \alpha A + S  
\end{equation}
in infinite volume is well-defined. Here $\alpha>0$ is a fixed constant. 
To that end, for $b>0$, define the following set of configurations:
\begin{equation*}
\mathscr X_b\coloneqq
\Big\{\omega= (\mathfrak p,\mathfrak r) \in \mathbb R^\mathbb Z\times \mathbb R^\mathbb Z:
\|\omega\|_b\coloneqq \sum_{j\in\mathbb Z} e^{-b|j|} \left(p_j^2 + r_j^2\right) < +\infty\Big\}.
\end{equation*}
A standard iteration argument proves that,
under the uniform Lipschitz condition on $V'$, the set $\mathscr X_b$ is left invariant
(see \cite{fritz1994stationary} and references within) so that the corresponding dynamics is well defined.
In this article, note that we do not need a uniform Lipschitz condition on $V'$, and we assume that the dynamics is well defined.
In what follows, we denote by $\mathscr X=\bigcup_{b>0}\mathscr X_b$ the state space of our chain of oscillators,
and let $\{ (\mathfrak r(t), \mathfrak p(t)); t\ge 0 \} $ be the Markov process on $\mathscr X$
with its infinitesimal generator given by \eqref{eq:gen}, where we omit the dependency on $n$
of the process to simplify the notation.

\subsection{Equilibrium Measures}
For the process $\{ (\mathfrak r(t), \mathfrak p(t)); t\ge 0 \} $,
the corresponding equilibrium Gibbs measure $\nu_n$
are parameterized by $\rho=(\beta,p,\tau)\in\mathbb R^3$,
with the respective components are the inverse temperature, velocity and tension. 
In addition, $\nu_n$ is a product measure whose common marginal is given by  
\begin{equation*}
\nu_n(r_j,p_j) 
= \sqrt{\frac{\beta}{2\pi}} 
e^{-(\beta/2) (p_j-p)^2}   
\frac{1}{Z_n(\beta,\tau)} e^{-\beta(V_n(r_j) -\tau r_j)} 
dr_j dp_j 
\end{equation*}
where
\begin{equation*}
Z_n(\beta,\tau)
= \int_{\mathbb R} 
e^{-\beta(V_n(\rho) -\tau \rho )} d\rho . 
\end{equation*}
Here, note that under Assumption \ref{asm:potential}, we can show the following estimate on the exponential moment with respect to the invariant measure $\nu_n$. 
Although this result can be shown analogously to \cite[Lemma 2.2]{gonccalves2023derivation}, a concise proof is presented in \cref{sec:static_estimate} for readers' convenience.

\begin{lemma}
\label{lem:static_estimate}
Fix $\beta>0$ and $p,\tau\in\mathbb{R}$. 
For any $\eta>0$, there exists $C_\eta>0$ and $n_c=n_c(\eta)>0$ such that 
\begin{equation}
\label{eq:uniform_moment_bound}
\sup_{n>n_c} E_{\nu_n} \big[e^{\eta(|r_j|+|p_j|)} \big] < C_\eta.
\end{equation}
\end{lemma}

In what follows, we assume that $n$ is sufficiently large so that the estimate \eqref{eq:uniform_moment_bound} holds with $\eta=2\eta_V$ where $\eta_V$ is the constant in Assumption \ref{asm:potential}. 
Note that $E_{\nu_n}[V_n'(r_j)]=\tau$ and $E_{\nu_n}[p_j]=p$. 
It is easy to prove that $\nu_n( \mathscr X_b) = 1$ for all $b>0$. 
Furthermore it is proven in  \cite{fritz1994stationary} that these are the only translation invariant stationary probabilities locally absolutely continuous.
In what follows, we set $p=0$ and $\tau=0$ for simplicity.   

Note that the operator $S$ (resp. $A$) is symmetric (resp. antisymmetric) with respect to the invariant measure $\nu_n$: 
\begin{equation*}
\langle f,Sg\rangle_{L^2(\nu_n)}
= \langle Sf,g\rangle_{L^2(\nu_n)}, \quad 
\langle f,Ag\rangle_{L^2(\nu_n)}
= -\langle Af,g\rangle_{L^2(\nu_n)}
\end{equation*}
for every local functions $f,g\in L^2(\nu_n)$. 
In addition, we introduce the norm $\| \cdot \|_{1,n}$ by 
\begin{equation*}
\begin{aligned}
\| f\|^2_{1} 
&=\langle - L f,f\rangle_{L^2(\nu_n)}
=\langle - S f,f\rangle_{L^2(\nu_n)}
=
\frac{\gamma}{4} \sum_{j\in\mathbb Z}
E_{\nu_n}\left[\big( f(\mathfrak r,\mathfrak p^{j,j+1})-f(\mathfrak r,\mathfrak p) \big)^2  \right]
\end{aligned}
\end{equation*}
for each local function $f\in L^2(\nu_n)$. 
Then, we define its dual norm $\| \cdot\|_{-1,n}$ by 
\begin{equation*}
\| f \|^2_{-1}
= \sup_{g}
\big\{2 \langle f, g\rangle_{L^2(\nu_n)}
- \| g \|^2_{1} \big\}  
\end{equation*}
where the supremum is taken over all local $L^2(\nu_n)$-functions.

We denote by ${D} ([0,T],\mathscr X) $ 
the space of the c\`adl\`ag (right-continuous and with left limits) trajectories taking values in $\mathscr X$. 
Let $\mathbb P_n$ be the probability measure in $ {D} ([0,T],\mathscr X)$ which is induced by $\nu_n$ and let $\mathbb E_n$ denote the expectation with respect to  $\mathbb P_n$. 
We will use the following well-known estimate on the $\|\cdot\|_{-1}$-norm. 
(See \cite[Lemma 2.4]{komorowski2012fluctuations}.)

\begin{proposition}
\label{prop:kipnis_varadhan_estiamte}
For any function $F:[0,T]\times L^2(\nu_n)$ such that $E_{\nu_n}[F(t,\cdot)]=0$ for any $t\in[0,T]$, we have that 
\begin{equation}
\label{eq:h-1}
\mathbb E_n \bigg[
\sup_{0\le t \le T} \bigg|\int_0^t F(s,\mathfrak r(s),\mathfrak p(s) ) ds \bigg|^2 \bigg]
\le 24 \int_0^T \| F(t,\cdot,\cdot) \|^2_{-1} dt
\end{equation}
\end{proposition}




\subsection{Main Result}
Before stating the main result let us introduce some notation and recall some previous results concerning different time scales and interactions. 
In what follows, we consider the chain of oscillators $\{ (\mathfrak r(t), \mathfrak p(t)): t \ge0 \}$
with generator $L_n = n^{\mathfrak a} L$,
starting from the invariant measure $\nu_n$, where we consider different time scales $\mathfrak a>0$. 
The process admits the following three locally conserved quantities: 
\begin{equation*}
r_j, \quad
p_j \quad \text{ and } \quad 
e_j :=  \frac{1}{2}p^2_j + V_n(r_j),
\end{equation*}
which correspond to the volume, momentum and energy, respectively. 

We define 
\begin{equation}
\label{eq:xi_newdef}
\xi_j^{\pm}  =\sqrt{c_2} r_{j+1} \pm \; p_j, \qquad \xi^0_j = e_j.
\end{equation}
The two linear combinations $ \xi_j^\pm$ are called \emph{phonon modes}, while
$\xi_j^0$ is called the \emph{energy mode} (also called \emph{heat mode}).
In what follows, writing the overline over variables means the centering
with respect to the invariant measure:
\begin{equation}
\overline{\xi^\sigma_j} 
= \xi^\sigma_j - E_{\nu_n}[\xi^\sigma_j].
\label{eq:barxi}
\end{equation}
We are interested in the evolution of the fluctuation fields of these modes, defined as
\begin{equation}
\label{eq:fluctuation_fields_definition}
\mathcal X^{\sigma,n}_t (\varphi)
= \frac{1}{\sqrt{n}} \sum_{j\in\mathbb Z} \overline{\xi^\sigma_j}(t) \varphi(j/n) 
\end{equation}
for each $\sigma = -1, 0 , 1$ and $\varphi\in\mathcal{S}(\mathbb R)$,
under different choice of the time scale $n^{\mathfrak a}$. 

In the Euler (or hyperbolic scaling) i.e. $\mathfrak a =1$, 
we have the following convergence of these fluctuations: 
\begin{equation}
\label{eq:euler}
\begin{split}
\lim_{n\to\infty} \mathcal X^{\sigma,n}_t (\varphi) 
=\mathcal X^{\sigma}_t (\varphi) =  \mathcal X^{\sigma}_0(\varphi_{\sigma \sqrt{c_2}\alpha t})
\end{split}
\end{equation}
for each $\sigma = -1, 0 ,1$, where $\varphi_z(x) = \varphi(x + z)$, and
$\mathcal X_0^\sigma$ is a Gaussian field with covariance
\begin{equation*}
\begin{split}
\mathbb E\left[ \left( \mathcal X^{\pm}_0 (\varphi)\right)^2\right] =
\frac 2\beta \|\varphi\|_{L^2}^2,
\qquad  \mathbb E\left[ \left( \mathcal X^{0}_0 (\varphi)\right)^2\right] =
\frac 3{\beta^2} \|\varphi\|_{L^2}^2,
\quad  \mathbb E\left[  \mathcal X^{\sigma}_0 (\varphi)
\mathcal X^{\sigma'}_0 (\varphi)\right] = 0,
\ \text{if} \ \sigma\neq\sigma'.
\end{split}
\end{equation*}
This means that, in this hyperbolic
scaling limit the phonon modes move deterministically in opposite directions
with the sound velocity of the harmonic lattice $\sqrt{c_2}\alpha$,
while the fluctuations of the energy mode do not evolve.
This convergence can be easily shown in the harmonic case
(cf. \cite{komorowski2016ballistic}) and extended to the small anharmonic perturbation for
any choice of $\varepsilon_n \to 0$.
With a different conservative noise, it can also be proven the non perturbative anharmonic
case ($\varepsilon_n = 1$) with the corresponding sound velocity \cite{olla2020equilibrium}.

The fluctuations of the energy mode evolve at a larger time scale, typically superdiffusive.
More precisely it is proven in \cite{jara2015superdiffusion}, in the harmonic case with $\mathfrak a= 3/2$
(i.e. at time scale $n^{3/2}$), that the space-time correlations of the energy mode evolve in agreement
to a 3/2-L\'evy type superdiffusion, i.e., the limiting distribution $\mathcal X_t^0$ should evolve following the linear SPDE 
\begin{equation}
\label{eq:levy2-3}
\begin{aligned}
\partial_t \mathcal X_t^0
= - 2^3\gamma^{-1/2}
(-\partial^2_x)^{3/4} \mathcal X_t^0 + \sqrt{ 2 \gamma^{-1/2} \beta^{-2}}
(-\partial^2_x)^{3/8}\dot W^0(t,x),
\end{aligned}
\end{equation}
where $\dot W^0(t,x)$ is a standard space-time white noise. 

In the purely harmonic case ($\ve_n = 0$) the
fluctuations of the phonon modes recentered on the Euler
evolution have diffusive behavior, i.e. with the choice $\mathfrak a=2$ 
we have \cite{komorowski2016ballistic}
\begin{equation}
\label{eq:phdiff}
\lim_{n\to \infty}
\mathcal X^{\pm,n}_t (\varphi_{\mp\sqrt{c_2}\alpha nt}) 
= \mathcal X^{\pm}_t (\varphi) 
\end{equation}
and the limiting fields $\mathcal X^{\pm}_t$ evolve with following the linear stochastic differential equation: 
\begin{equation}
\label{eq:ou}
\partial_t \mathcal X^{\pm}_t
=  \frac{\gamma}{4} \partial^2_x \mathcal X^{\pm}_t
+ \sqrt{\beta^{-1} \gamma} \partial_x \dot W_\pm(t,x),
\end{equation}
where $\dot W_+(t,x)$ and $\dot W_-(t,x)$ are two independent space-time white noise.
Notice that in the purely harmonic case
this diffusive behavior of the recentered fluctuations of the phonon modes are entirely due to the presence of the noise in the dynamics: the corresponding
deterministic dynamics will not have such fluctuations.

On the other hand, anharmonic terms in the interaction change this behavior, in particular if a cubic term
is present. 

Recall the definition and properties
of the energy solution of a Burgers equation stated in \cref{sec-ssbe}.
Denote by $T_{\mp n t \sqrt{c_2}\alpha} \mathcal X^{\pm,n}_t$ the recentered fluctuations of the phonon fields, i.e. 
$$
T_{\mp n t \sqrt{c_2}\alpha} \mathcal X^{\pm,n}_t (\varphi) =
\mathcal X^{\pm,n}_t(\varphi_{\mp n t \sqrt{c_2}\alpha}). 
$$
Then, our main theorem is stated as follows.

\begin{theorem}
\label{thm:sbe_derivation_from_chain}
Assume $\mathfrak a=2$ and $\varepsilon_n= n^{-1/2}$. Then the pair of the fluctuation fields
$T_{\mp n t \sqrt{c_2}\alpha} \mathcal X^{\pm,n}_t $
converges as $n\to\infty$ in distribution in the space $D([0,T],\mathcal S'(\mathbb R)^2)$
to $\{(u^+_t,u^-_t): t\in [0,T]\}$,
which is a pair of the stationary energy solution of the stochastic Burgers equation
\begin{equation}\label{eq:drift-sbe}
\partial_t u^\sigma
=\frac{\gamma}{4}\partial_x^2 u^\sigma
+\sigma\frac{\alpha c_3}{8c_2^2}\partial_x (u^\sigma)^2 
+ \sigma\alpha D_V \partial_x u^\sigma
+ \sqrt{\gamma\beta^{-1}}
\partial_x \dot{W}^\sigma,
\end{equation} 
for each $\sigma\in \{+,-\}$, where 
\begin{equation}
\label{eq:DV} 
D_V =
\Add{\frac{c_2c_4 - c_3^2}{4c_2^{5/2} \beta}}.
\end{equation}
and $\dot{W}^\sigma=\dot{W}^\sigma(t,x), \sigma = \pm 1$,
are independent 
space-time white-noises. 
\end{theorem}

Since $\lim_{n\to\infty}\mathrm{Var}_{\nu_n}[\xi^\pm_j]= 2\beta^{-1}$,
the limiting equation \eqref{eq:drift-sbe}
clearly satisfies the condition for stationarity (see \textbf{(S)} in \cref{sec-ssbe}). 
Notice that the linear drift term in \eqref{eq:drift-sbe}
indicate that the \emph{effective sound velocity} is now
\begin{equation}
\label{eq:moving_frame_modification}
\sigma\sqrt{c_2}\alpha(1 + \frac 1n D_V) .  
\end{equation}

\Add{
\begin{remark}
Theorem~\ref{thm:sbe_derivation_from_chain} holds also when $c_3 = 0$ (pure quartic anharmonicity). The equation~\eqref{eq:drift-sbe} is linear in this case, but still there is a drift $\sigma \alpha c_4/(4c_2^{3/2}\beta)$ due to the presence of the quartic anharmonicity. 
\end{remark}
}

\begin{remark}
\label{rem-b}
The proof can be extended to stronger anharmonicity, with $\ve_n = n^{-\mathfrak b}$ with $\mathfrak b\in (1/4,1/2)$, with the price to slow down the Hamiltonian part of the dynamics, i.e. substituting
$\alpha$ with \Erase{$\alpha_n = n^{-\mathfrak a} \alpha$, $\mathfrak a\in (7/4,2)$}\Add{$\alpha_n = n^{\mathfrak c-2} \alpha$, $\mathfrak c\in (7/4,2)$ and we take $\mathfrak a=2$. Note that $\mathfrak c$ is the exponent of the intensity of the Hamiltonian part.}  
Then, choosing
\begin{equation}
\label{eq:nonlinear_term_critical_line}
\Erase{\mathfrak a = \mathfrak b + 3/2}
\Add{\mathfrak c=\mathfrak b+3/2}
\end{equation}
we obtain the same limit equations \eqref{eq:drift-sbe} but with $D_V = 0$. 
For the sub-critical case, i.e., when $\mathfrak c< \mathfrak b+3/2$, we expect diffusive behavior like \eqref{eq:ou}.  
On the other hand, for the case \Erase{$\hat{\beta}\in (0,1/4]$}\Add{$\mathfrak b \in (0,1/4]$}, 
we need to handle higher-order terms in the Taylor expansion of the function $V_n$.
Although we decided here to avoid those cases, leaving them as a future work,
we expect that the limiting equation will be analogous to the main theorem - when on (resp. below) the critical line \eqref{eq:nonlinear_term_critical_line} the limit would be given by the SBE (resp. stochastic heat equation).   
The whole picture of our main result as well as the conjecture
that we described above is summarized in \cref{fig:phonon_fluctuation}. 
\end{remark}

\begin{figure}[htb!]
\begin{center}
\begin{tikzpicture}[scale=0.3]
\draw (0,25) node[left]{\Erase{$\mathfrak a$}\Add{$\mathfrak c$}};
\draw (25,0) node[below]{$\mathfrak b$};
\draw (20,0) node[below]{$1/2$};
\draw (10,0) node[below]{$1/4$};
\draw (0,0) node[left]{$1$};
\draw (0,0) node[below]{$0$};
\draw (0,10) node[left] {$3/2$};
\draw (0,15) node[left] {$7/4$};
\draw (0,20) node[left]{$2$};
\fill[light-gray] (0,10) -- (20,20) -- (25,20) -- (25,25) -- (0,25) -- cycle;
\fill[fill=cyan, fill opacity=0.15] (0,0) -- (25,0) -- (25,20) -- (20,20) -- (10,15)--cycle;
\path[pattern=north west lines, pattern color=blue, fill opacity=0.9] (0,0) -- (10,15) -- (0,10)--cycle;
\draw[-,=latex, magenta, ultra thick] (10,15) -- (20,20) node[midway,above,sloped] {\textbf{SBE}};
\draw[-,=latex, magenta, dashed, ultra thick] (0,10) -- (10,15) node[midway,above,sloped] {\textbf{SBE?}};
\draw[-,=latex, cyan, ultra thick] (20,20) -- (25,20) node[midway,above,sloped] {\textbf{SHE}};
\node[circle,fill=magenta,inner sep=1.0mm] at (20,20) {};
\node[circle,fill=magenta,inner sep=1.0mm] at (10,15) {};
\node[circle,fill=white, inner sep=0.6mm] at (10,15) {};
\node[] at (15,10) {\textcolor{cyan}{\textbf{SHE}}};
\node[] at (2.5,9) {\textcolor{blue}{\textbf{SHE?}}};
\draw[-,=latex, dashed] (20,-0.1) -- (20,19.5);
\draw[-,=latex, dashed] (0,0) -- (19.5,19.5);
\draw[-,=latex, dashed] (10,0) -- (10,14.5);
\draw[-,=latex, blue, dashed, ultra thick] (0,0) -- (9.5,14.5);
\draw[-,=latex, dashed] (0,15) -- (9.5,15); 
\draw[-,=latex, dashed] (0,20) -- (19.5,20);
\draw[->,>=latex] (0,0) -- (26,0);
\draw[->,>=latex] (0,0) -- (0,26);
\end{tikzpicture}
\end{center}
\caption{Expected fluctuations for the two phonon modes,
for each choice of scaling \Erase{$\alpha_n=n^{\mathfrak a-2} \alpha$}\Add{$\alpha_n=n^{\mathfrak c-2} \alpha$} and $\varepsilon_n= n^{-\mathfrak b}$.  
We only prove here the case \Erase{$\mathfrak a = 2$}\Add{$\mathfrak c=2$} (diffusive scaling) and $\mathfrak b= 1/2$, but the proof can be extended along the red continuous line up to $\mathfrak b>1/4$ and \Erase{$\mathfrak a > 7/4$}\Add{$\mathfrak c>7/4$}.}
\label{fig:phonon_fluctuation}
\end{figure}
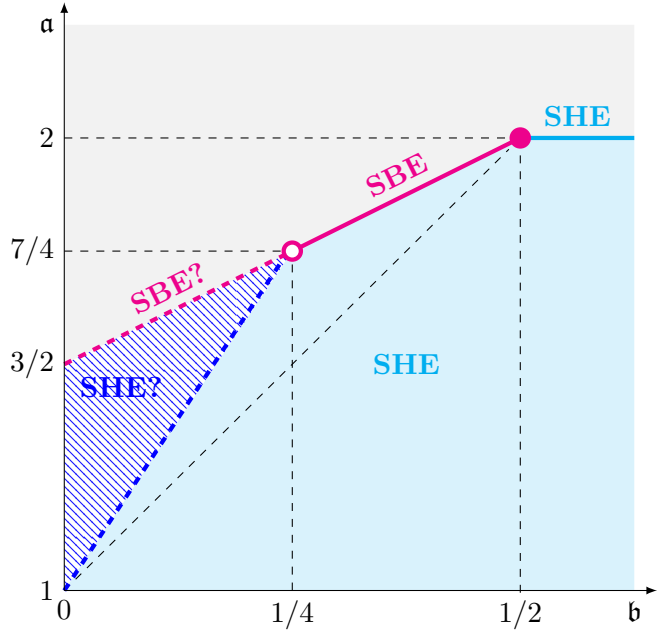

\subsection{Some equilibrium relations}
\label{sec:some-equil-relat}

In the following these relations on  equilibrium expectations will be useful:
\begin{equation}
\label{eq:8}
\begin{split}
E_{\nu_n}[r_j V'_n(r_j)] = \frac 1\beta, \quad
E_{\nu_n}[r_j^2 V'_n(r_j)]  = \frac{2}{\beta} E_{\nu_n}[r_j],
  \end{split}
\end{equation}
and by expanding $V_n$ in the first identity of \eqref{eq:8} we obtain
\begin{equation}
\label{eq:r2}
E_{\nu_n}[r_j^2]
= \frac{1}{c_2\beta}
-\ve_n \frac{c_3}{2 c_2} E_{\nu_n}[r_j^3]
- \ve_n^2 \frac{c_4}{3! c_2} E_{\nu_n}[r_j^4] + O(\ve_n^3).
\end{equation}
Then, since $ E_{\nu_n}[V'_n(r_j)]= 0$,
\begin{equation}
\label{eq:r}
E_{\nu_n}[r_j] 
= -\ve_n \frac{c_3}{2c_2} E_{\nu_n}[r_j^2] + O(\ve_n^3)
= -\ve_n \frac{c_3}{2c_2^2\beta} + O(\ve_n^3).
\end{equation}
Furthermore
\begin{equation*}
E_{\nu_n}[r_j^3 V'_n(r_j)]
= 3 E_{\nu_n}[r_j^2] \frac{1}{\beta}
= \frac{3}{c_2\beta^2}   + O(\ve_n^2),
\end{equation*}
that gives after expansion
 \begin{equation*}
E_{\nu_n}[r_j^4]
= \frac{3}{c_2^2\beta^2} + O(\ve_n^2).
\end{equation*}
the second identity of \eqref{eq:8}, after expansion of $V_n$ we get
 \begin{equation*}
c_2 E_{\nu_n}[r_j^3] - \frac{2}{\beta} E_{\nu_n}[r_j] 
= - \ve_n \frac{c_3}{2}    E_{\nu_n}[r_j^4]+ O(\ve_n^3)
= - \ve_n \frac{3c_3}{2c_2^2 \beta^2}  + O(\ve_n^3),
 \end{equation*}
i.e.,
\begin{equation*}
c_2 E_{\nu_n}[r_j^3] 
= - \ve_n \frac{5 c_3}{2c_2^2 \beta^2}  + O(\ve_n^3).
 \end{equation*}

\section{Martingale decomposition} 
\label{sec:sketch}
Here we give a sketch of the proof of the main result.
In the following we set 
$ v^\sigma_n = \sigma\sqrt{c_2}\alpha n$.
Recall we have defined $\varphi_z(x) = \varphi(x + z)$ for all $x,z\in \mathbb R$.

We will see that in the time evolution of the phonon modes a large divergent term arises from the cubic term of the interaction, i.e., a term like $(c_3/2)r_j^2$. In order to balance this term we need to add a small \emph{corrector} to the definition of the phonon modes proportional to the energy mode (see \eqref{eq:I} and the following). 
For that purpose, let us consider a weak perturbation of the
phonon modes by the energy:
\begin{equation}
\label{eq:xi_definition}
\tilde\xi_j^\sigma = \sqrt{c_2} r_{j+1} + \sigma p_j + \varepsilon_n \mathfrak u e_j
\end{equation}
where 
\begin{equation}
\label{eq:coupling_constant_choice}
\mathfrak u = \frac{c_3}{2c_2^{3/2}} . 
\end{equation}
\Add{Notice that with this choice of $\mathfrak u$, by \eqref{eq:r}, $E_{\nu_n}[\tilde\xi^\sigma_j] = O(\ve_n^3)$.}

Correspondingly define for each $\varphi\in\mathcal S(\mathbb R)$
\begin{equation*}
\tilde{\mathcal{X}}^{\sigma,n}_t(\varphi)
= \frac{1}{\sqrt{n}} \sum_{j\in\mathbb Z}
(\tilde\xi^\sigma_j(t) - E_{\nu_n}[\tilde\xi^\sigma_j]\big)  \varphi(\tfrac{j}{n}).
\end{equation*}
Note that the field $\tilde{\mathcal X}^{\sigma,n}_\cdot$ is a small
perturbation of the one associated to the phonon mode $\mathcal X^{\sigma,n}_\cdot$ and,
as  a consequence of the following lemma, these modified fields should be identical in the limit $n\to\infty$. 

\begin{lemma}\label{lem:corr}
The process $\{\mathcal{X}^{\sigma,n}_t(\varphi_{v^\sigma_nt})
- \widetilde{\mathcal{X}}^{\sigma,n}_t(\varphi_{v^\sigma_nt})\}_{0\le t \le T}$
converges in law to $0$ in $D([0,T], \mathbb R)$. 
\end{lemma}
The proof of Lemma \ref{lem:corr} is postponed to appendix \cref{sec:proof-lemma-refl}. 
Now, let us consider the time evolution of the modified field $\widetilde{\mathcal X}^{\sigma,n}_\cdot$.  
Our starting point is to obtain a martingale decomposition in the following way.

\begin{lemma}
\label{lem:martingale_decomposition_main}
For each $\sigma\in \{+,-\}$ and for any $\varphi\in\mathcal S(\mathbb R)$, we have that 
\begin{equation}
\label{eq:martingale_decomposition}
\begin{aligned}
\widetilde{\mathcal{X}}^{\sigma,n}_t(\varphi _{v^\sigma_nt})
&= \widetilde{\mathcal{X}}^{\sigma,n}_0(\varphi)
+\frac{\gamma}{4} \int_0^t
\widetilde{\mathcal{X}}^{\sigma,n}_s(\partial_x^2 \varphi _{v^\sigma_ns}) ds \\
&\quad-\sigma\alpha\frac{c_3}{8c_2^2} 
\int_0^t\sum_{j\in\mathbb Z}
\overline{\xi}^{\sigma}_j(s) \overline{\xi}^\sigma_{j+1}(s)
\partial_x \varphi(\tfrac{j}{n}+ {v^\sigma_ns}) ds\\
&\quad+ \sigma \alpha \frac{c_3^2-2c_2c_4}{4 c_2^{5/2}\beta}
\int_0^t \widetilde{\mathcal{X}}^{\sigma,n}_s(\partial_x \varphi _{v^\sigma_ns})ds 
+ \mathcal{M}^{\sigma,n}_t(\varphi)
+ \mathcal{R}^{\sigma,n}_t(\varphi)
\end{aligned}
\end{equation}
where $\mathcal M^{\sigma,n}_\cdot(\varphi)$ is a martingale with respect to the natural filtration whose quadratic variation satisfies
\begin{equation}
\label{eq:quadratic_variation_simple}
\begin{aligned}
\langle \mathcal{M}^{\sigma,n}(\varphi)\rangle_t
&= \int_0^t \frac{1}{2n}
\sum_{j\in\mathbb Z}
\gamma (p_j - p_{j+1})^2 (s)
\big(\partial_x\varphi(\tfrac{j}{n}+{v^\sigma_ns})\big)^2 ds
+ \widetilde{\mathcal R}^{\sigma,n}_t(\varphi)
\end{aligned}
\end{equation}
where the reminder terms $\mathcal{R}^{\sigma,n}_\cdot$ and
$\widetilde{\mathcal{R}}^{\sigma,n}_\cdot$ vanish in law as $n\to\infty$.  
\end{lemma}

Looking at the resulting martingale decomposition \eqref{eq:martingale_decomposition},
we can naturally expect that the second term on the  right-hand side gives rise to the viscosity
term in the limiting equation \eqref{eq:drift-sbe}, the third term will give the nonlinearity and
the fourth term the drift.


We will show in \cref{sec:tightness} that each term in the decomposition
\eqref{eq:martingale_decomposition} is tight, and then we identify the limit points in
\cref{sec:identification_limit_points} to see that the aforementioned heuristics are indeed the case. 
In the forthcoming parts of this section is devoted to prove \cref{lem:martingale_decomposition_main}.

By Dynkin's formula, the processes  
\begin{equation*}
\mathcal{M}^{\sigma,n}_t (\varphi)
= \widetilde{\mathcal{X}}^{\sigma,n}_t (\varphi_{v^\sigma_nt})
-\widetilde{\mathcal{X}}^{\sigma,n}_0(\varphi)
-\int_0^t(\partial_s + L_n) \widetilde{\mathcal{X}}^{\sigma,n}_s(\varphi_{v^\sigma_n s}) ds 
\end{equation*}
and $\mathcal{M}^{\sigma,n}_t(\varphi)^2 - \langle \mathcal{M}^{\sigma,n} (\varphi)\rangle_t$
are martingale where 
\begin{equation}
\label{eq:quadratic_variation}
\begin{aligned}
\langle \mathcal{M}^{\sigma,n}(\varphi)\rangle_t
= \int_0^t \Gamma^{\sigma,n}_s(\varphi)ds 
\end{aligned}
\end{equation}
with the carr\'e du champs operator $\Gamma^{\sigma,n}_\cdot$ given by 
\begin{equation}
\label{eq:carre_du_champs}
\Gamma^{\sigma,n}_t(\varphi) 
= L_n \widetilde{\mathcal{X}}^{\sigma,n}_t(\varphi_{v^\sigma_nt})^2 
-2 \widetilde{\mathcal{X}}^{\sigma,n}_t(\varphi_{v^\sigma_nt})
L_n \widetilde{\mathcal{X}}^{\sigma,n}_t(\varphi_{v^\sigma_nt}) 
\end{equation}
for each $\varphi\in\mathcal S(\mathbb R)$ and $t\in [0,T]$. 
A simple computation shows that carr\'e du champs $\Gamma^{\sigma,n}_\cdot$ is computed as follows: 
\begin{equation}
\label{eq:carre_du_champs_computation}
\Gamma^{\sigma,n}_t(\varphi)
= \frac{\gamma}{2n} \sum_{j\in\mathbb Z} (\Xi^{p,\sigma}_j-\Xi^{p,\sigma}_{j+1})^2(t)
(\nabla^n \varphi\big(\tfrac{j}{n} + v^\sigma_nt)\big)^2  
\end{equation}
for each $\varphi\in\mathcal S(\mathbb R)$, where we introduced the variables 
\begin{equation}\label{eq:Xip}
\Xi^{p,\sigma}_j
= \sigma p_j + \frac{\varepsilon_n \mathfrak u}{2} p_j^2. 
\end{equation}
Then, it is not hard to see that, by dropping terms with $\varepsilon_n$ in \eqref{eq:carre_du_champs_computation},
the main term of the quadratic variation is given by \eqref{eq:quadratic_variation_simple}. 

Thus, to obtain the martingale decomposition in the form of \eqref{eq:martingale_decomposition}, what is left is to compute the action of the generator. 
Hereinafter we split the generator into symmetric and antisymmetric parts and compute the action separately.



\subsection{Identification of the nonlinear part of the drift}
\label{sec:ident-non-line}
Now, to proceed, let us compute the action of the antisymmetric (i.e., Hamiltonian) part of the generator.
For each $\eta=r,p,e$, we define instantaneous current $J^\eta_{j,j+1}$ of $\eta$ by 
\begin{equation*}
A \eta_j = J^\eta_{j-1,j} -J^\eta_{j,j+1}. 
\end{equation*}
Then, we compute
$Ar_{j}=\nabla^+p_{j-1}$, $Ap_j=\nabla^+ V'_n(r_j)$ and
$Ae_j= \nabla^+(p_{j-1}V'_n(r_j))$ so that 
\begin{equation}
\label{eq:current_computation}
J^r_{j,j+1}= - p_{j} ,\quad
J^p_{j,j+1}=- V_n'(r_{j+1}) ,\quad
J^e_{j,j+1}=-p_{j} V_n'(r_{j+1}). 
\end{equation}
Moreover, let
\begin{equation*}
\begin{aligned}
J^{\tilde\xi^\sigma}_{j,j+1}
&= \sqrt{c_2} J^r_{j+1,j+2} 
+\sigma J^p_{j,j+1}
+\varepsilon_n\mathfrak u J^e_{j,j+1} \\
&= - \sqrt{c_2} p_{j+1} - \sigma  V_n'(r_{j+1}) - \varepsilon_n \mathfrak u p_{j} V_n'(r_{j+1})
\end{aligned}
\end{equation*}
be the instantaneous current of $\tilde\xi^\sigma$. 
Then, by summation-by-parts, we have  
\begin{equation*}
\begin{aligned}
(\partial_t + \alpha n^2 A) \widetilde{\mathcal{X}}^{\sigma,n}_t (\varphi_{v^\sigma_nt})
= {\alpha}{n^{1/2}} 
\sum_{j\in\mathbb Z}
J^{\tilde\xi^\sigma}_{j,j+1} \nabla^n \varphi (\tfrac{j}{n}+v^\sigma_n t) 
+ \frac{v^\sigma_n}{n^{1/2}} \sum_{j\in\mathbb Z}
\tilde\xi^\sigma_j(t) \partial_x \varphi(\tfrac{j}{n}+v^\sigma_n t) .
\end{aligned}
\end{equation*}
As we shall see below, it turns out that 
\begin{equation}
\label{eq:antisymmetric_current_dominant}
W_j = - 
\mathfrak u p_j\overline{r}_{j+1} \ven 
\end{equation}
remains as a dominant term when $n\to\infty$ in the right-hand side of the last display, provided $\mathfrak u$ is chosen by \eqref{eq:coupling_constant_choice}.
Then, we can see that the last display gives rises to the quadratic field $\mathcal B^{\sigma,n}_\cdot$ in the martingale decomposition \eqref{eq:martingale_decomposition}. 
Indeed, firstly note that we have the following identity: 
\begin{equation*}
\begin{split}
(\overline\xi^+_j)^2 - (\overline\xi^-_j)^2 = 4 \sqrt{c_2} p_j\overline{r}_{j+1}.
\end{split}
\end{equation*}
Then, we have that 
\begin{equation}\label{eq:W}
\begin{split}
\alpha n^{1/2} \sum_{j\in \mathbb Z} W_j  \partial_x\varphi(\tfrac{j}{n}+v^\sigma_nt)  
&= -\frac{c_3}{2c_2^{3/2}} \alpha 
\sum_{j\in \mathbb Z} p_j\overline{r}_{j+1}
\partial_x\varphi(\tfrac{j}{n}+v^\sigma_nt) 
+ O(\varepsilon_n ) \\
&= -\frac{c_3}{8c_2^2} 
\alpha
\sum_{j\in \mathbb Z}
( (\overline\xi^+_j)^2 - (\overline\xi^-_j)^2) \varphi(\tfrac{j}{n}+v^\sigma_nt)  
\end{split}
\end{equation}
for each $t\in[0,T]$.
The last display is reminiscent of the nonlinear term in the stochastic Burgers equation.  
However, when we take a correct velocity for each mode $\sigma=\pm$, the nonlinear term in the last display concerning the other mode, which is observed in a wrong (divergent) moving frame, would be vanishing, and thus the limiting equation is decoupled with the same absolute value of coefficient in front of the nonlinear term for each mode.  
This is a consequence of \cref{lem:quadratic_field_with_wrong_velocity} - a version of Riemann-Lebesgue lemma, which will be justified later in~\cref{sec:riemann_lebesgue_estimates}.

In what follows, let us see that \eqref{eq:antisymmetric_current_dominant} is indeed the dominant term and the other terms do not contribute in the limit.  
To that end, let us replace the discrete derivative by the continuous one.
Note that 
\begin{equation*}
\Big| \nabla^n \varphi(\tfrac{j}{n})
- \partial_x \varphi(\tfrac{j}{n})
- \frac{1}{2n} \partial_x^2 \varphi(\tfrac{j}{n}) \Big|
\le \frac{1}{6n^2} \sup_{0\le a \le 1}
\big| \partial_x^3\varphi(\tfrac{j+a}{n}) \big|,
\end{equation*}
from which we deduce  
\begin{equation*}
(\partial_t + n^2 \alpha A) \widetilde{\mathcal{X}}^{\sigma,n}_t (\varphi_{v^\sigma_nt})    
= I^\sigma + II^\sigma + O\Big(\frac{1}{n}\Big) 
\end{equation*}
where 
\begin{equation*}
\begin{aligned}
&I^\sigma = 
\alpha n^{1/2} 
\sum_{j\in\mathbb Z} 
\Big( J^{\tilde\xi^\sigma}_{j,j+1}(t)
+ \sigma \sqrt{c_2} 
\tilde\xi^\sigma_j(t) \Big) 
\partial_x \varphi(\tfrac{j}{n}+v^\sigma_nt)  , \\ 
&II^\sigma = \frac{\alpha}{2 n^{1/2}} 
\sum_{j\in\mathbb Z} 
J^{\tilde\xi^\sigma}_{j,j+1}(t) \partial_x^2 \varphi(\tfrac{j}{n}+v^\sigma_nt)  .
\end{aligned}
\end{equation*}
Now, by Taylor's theorem, note that 
\begin{equation*} 
V_n(r)= \frac{c_2}{2}r^2 
+ \frac{c_3}{3!} \varepsilon_n r^3
+ \frac{c_4}{4!} \varepsilon_n^2 r^4
+ O(\varepsilon_n^3)
\end{equation*} 
and
\begin{equation*}
V_n'(r)
=c_2 r 
+ \frac{c_3}{2} \varepsilon_n r^2 
+ \frac{c_4}{3!} \varepsilon_n^2 r^3
+ O(\varepsilon_n^3).
\end{equation*}
In the last line, the reminder term $O(\varepsilon_n^3)$
is absolutely bounded by $\varepsilon_n^3 e^{\eta_V|\eta|}$
where $\eta_V$ is the constant in Assumption \ref{asm:potential}. 
In particular, note that this reminder term satisfies the bound 
\begin{equation}
\label{eq:err}
\begin{aligned}
\mathbb E_n \bigg[\sup_{0\le t\le T}\bigg| n^{1/2} 
\sum_{j\in\mathbb Z}
\big( O(\varepsilon_n^3) - E_{\nu_n}[O(\varepsilon_n^3)]\big)
\varphi(\tfrac{j}{n}) \bigg|^2 \bigg]
\lesssim 
n^2 \varepsilon_n^6 
\| \varphi\|^2_{L^2(\mathbb R)} = \frac 1n \| \varphi\|^2_{L^2(\mathbb R)},
\end{aligned}
\end{equation}
for any $\varphi\in\mathcal S(\mathbb R)$. Here one can clearly see that the moving frame does not matter in the bound \eqref{eq:err}. 

Now, recall that $E_{\nu_n}[V_n'(r_j)]=\tau=0$. 
It is not hard to see that 
\begin{equation}
\label{eq:r_moment_computation}
E_{\nu_n}[r_j^2]
= \frac{1}{c_2\beta} + O(\varepsilon_n^2), \quad
E_{\nu_n}[r_j]= 
-\frac{c_3}{2c_2^2\beta}\varepsilon_n + O(\varepsilon_n^3) .
\end{equation}
Thus, we rewrite $V'_n$ as 
\begin{equation}
\label{eq:5}
\begin{aligned}
V_n'(r_j) 
&= c_2  r_j + \frac{c_3}{2} \varepsilon_n {r_j^2} 
+ \frac{c_4}{3!} \varepsilon_n^2 r_j^3
+ O(\varepsilon_n^3) \\
&=c_2 \overline{r}_j  
- \frac{c_3}{2\beta c_2}\varepsilon_n
+ \frac{c_3}{2} \varepsilon_n \overline{r_j^2} + \frac{c_3}{2\beta c_2}\varepsilon_n
+ \frac{c_4}{3!} \varepsilon_n^2 r_j^3+ O(\varepsilon_n^3) \\
&= 
c_2 \overline{r_j}  
+ \frac{c_3}{2} \varepsilon_n \overline{r_j^2}
+ \frac{c_4}{3!} \varepsilon_n^2 r_j^3
+ O(\varepsilon_n^3) .
\end{aligned}
\end{equation}
Consequently,  
\begin{equation*}
\begin{aligned}
J^{\tilde\xi^\sigma}_{j,j+1}
&= - \sqrt{c_2}  p_{j+1}
- \sigma \Big(c_2 \overline{r_{j+1}} +
\frac{c_3}{2}\varepsilon_n \overline{r_{j+1}^2}
+ \frac{c_4}{6} \varepsilon_n^2 r_{j+1}^3 \Big) \\
&\quad- \varepsilon_n\mathfrak u p_j \Big(c_2 \overline{r_{j+1}}
+\frac{c_3}{2} \varepsilon_n \overline{r_{j+1}^2} \Big)
+ O(\varepsilon_n^3) \\
&= - \sqrt{c_2} p_{j+1}
-\sigma c_2 \overline{r_{j+1}} 
- \Big( \sigma\frac{c_3}{2} \overline{r_{j+1}^2}
+ \mathfrak u c_2 p_j \overline{r_{j+1}} \Big)\varepsilon_n \\
&\quad- \Big( \sigma\frac{c_4}{6} r_{j+1}^3 
+ \mathfrak u\frac{c_3}{2}  p_j \overline{r_{j+1}^2} \Big)\varepsilon_n^2
+ O(\varepsilon_n^3) 
\end{aligned}
\end{equation*}
and 
\begin{equation*}
\tilde{\xi}^\sigma_j - E_{\nu_n}[\tilde{\xi}^\sigma_j]
= \sqrt{c_2} \overline{r_{j+1}} 
+ \sigma p_j 
+\mathfrak u
\Big(\frac{1}{2}\overline{p_j^2} +\frac{1}{2}c_2 \overline{r_j^2} \Big) \varepsilon_n
+ \frac{c_3}{6} \mathfrak u \overline{r_j^3} \varepsilon_n^2
+ O(\varepsilon_n^3).  
\end{equation*}
With these expansions at hand, we can write 
\begin{equation}
\label{eq:I}
\begin{aligned}
I^\sigma = \sum_{k=0,1,2,3} \alpha n^{1/2}
\sum_{j\in\mathbb Z}
F^{\sigma,k}_j(\mathfrak r,\mathfrak p)
\partial_x \varphi(\tfrac{j}{n}+v^\sigma_nt) 
+ O(n^{-1}), 
\end{aligned}
\end{equation}
where the error $O(n^{-1})$ is negligible in the sense of \eqref{eq:err},
and $F^{\sigma,k}_j=F^{\sigma,k}_j(\mathfrak r, \mathfrak p)$ are defined by 
\begin{equation*}
\begin{aligned}
F^{\sigma,0}_j&= - \sigma c_2 (r_{j+1}- r_j) ,\\
F^{\sigma,1}_j &= - \sqrt{c_2} p_{j+1}
- \sigma c_2 \overline r_j
+ \frac{nv^\sigma_n}{\alpha n^2} (\sqrt{c_2} \overline r_j + \sigma p_j) =
- \sqrt{c_2} p_{j+1}
- \sigma c_2 \overline r_j
+ \sigma \sqrt{c_2} (\sqrt{c_2} \overline r_j + \sigma p_j) \\
&=\sqrt{c_2}( p_{j+1}-p_j),  \\
F^{\sigma,2}_j &= - \Big( \sigma\frac{c_3}{2} \overline{r_{j+1}^2}
+ \mathfrak u p_j \overline r_{j+1} \Big)\varepsilon_n 
+  \frac{nv^\sigma_n}{2\alpha n^2} \mathfrak u (c_2 \overline{r_j^2} + \overline{p_j^2})
\varepsilon_n \\
&= - \Big( \sigma\frac{c_3}{2} \overline{r_{j+1}^2}
+ \mathfrak u p_j \overline r_{j+1} \Big)\varepsilon_n 
+  \sigma \frac{\sqrt{c_2}}{2} \mathfrak u (c_2 \overline{r_j^2} + \overline{p_j^2})
\varepsilon_n ,\\
F^{\sigma,3}_j &=  - \Big( \sigma\frac{c_4}{6} \overline{r_{j+1}^3} 
+ \frac{c_3}{2}\mathfrak u p_j \overline{r_{j+1}^2} \Big)\varepsilon_n^2
+ \sigma \sqrt{c_2} \frac{c_3}{6} \mathfrak u \overline{r_j^3} \varepsilon_n^2.
\end{aligned}
\end{equation*}
Above, we used the fact that we can freely add any constant to $F^k_j$ since
telescopic sums vanish in \eqref{eq:I}.  

First, regarding the contribution of $F_j^{\sigma,0}$, by a summation-by-parts, we have that 
\begin{equation}
\label{eq:F0_second_derivative_from_asymmetric}
\alpha n^{1/2} 
\sum_{j\in\mathbb Z}
F^{\sigma,0}_j(\mathfrak r,\mathfrak p) \partial_x \varphi(\tfrac{j}{n}+v^\sigma_nt) 
= \frac{\sigma \alpha c_2}{n^{1/2}} 
\sum_{j\in\mathbb Z} \overline r_{j+1} \nabla^n \partial_x \varphi(\tfrac{j}{n}+v^\sigma_nt).
\end{equation}
The contribution of $F^0_j$ seems to survive in the limit at first glance.
It turns out, however, that sum of \eqref{eq:F0_second_derivative_from_asymmetric}
and $II$ is negligible, according to the fact that the fluctuation fields with wrong velocity are negligible,
which is a consequence of \cref{lem:field_with_wrong_velocity}.
We have a similar argument for $F_j^{\sigma,1}$.


Next, we consider the contribution of $F^{\sigma,2}_j$ that we write as
\begin{equation}
\label{eq:F2}
\begin{split}
F^{\sigma,2}_j
&= - \sigma \frac{c_3}{2} (r_{j+1}^2-r_j^2)\varepsilon_n 
+ \Big[ \Big(- \sigma\frac{c_3}{2} 
+ \frac{\sigma}{2} \mathfrak u c_2^{3/2} \Big) \overline{r_{j}^2}
+ \frac{\sigma \sqrt{c_2}}{2} \mathfrak u\overline{p_j^2} \Big] \varepsilon_n 
- \mathfrak u p_j \overline r_{j+1} \varepsilon_n \\
&= - \sigma \frac{c_3}{2} (r_{j+1}^2-r_j^2)\varepsilon_n 
+\sigma\frac{c_3}{4c_2}  \Big[ - c_2\overline{r_{j}^2}
+ \overline{p_j^2} \Big] \varepsilon_n 
- \mathfrak u p_j \overline r_{j+1} \varepsilon_n \\
\end{split}
\end{equation}
Using summation-by-parts, it is not difficult to see that the first term in the right-hand side of \eqref{eq:F2} does not contribute to the limit. 
To make the second term of the last display negligible, the following estimate regarding equipartition of energy is crucial. 

\begin{lemma}[Equipartition of energy] 
\label{lem-equipart}
For each mode $\sigma\in \{+,-\}$, we have that 
\begin{equation*}
\limsup_{n\to\infty} \mathbb E_n \bigg[ 
\sup_{0\le t \le T} \bigg|
\int_0^t \sum_{j\in \mathbb Z} \big( c_2\overline{r_j^2}(s) - \overline{p_j^2}(s)
+\Add{\ve_n \frac{c_3}{c_2\beta}\overline{r_j}(s)}  \big)
\partial_x\varphi(\tfrac{j}{n}+v^\sigma_ns)  ds \bigg|^2 \bigg] = 0. 
\end{equation*}
\end{lemma}

The proof of \cref{lem-equipart} is postponed to \cref{sec:energy-fluctuations}.  
Now, we are left with the main term $W_j$ (see \eqref{eq:antisymmetric_current_dominant}), which in fact is written as the third term in \eqref{eq:martingale_decomposition} and this gives rise to the nonlinear term in the limiting equation as we explained in the beginning. 
\Add{Additionally, notice that after an application of \cref{lem-equipart}, we are left with the term 
\begin{equation}
\label{eq:linear_term_from_equipartition}
\sigma \frac{c_3^2}{4c_2^2\beta} \varepsilon_n^2 
\overline r_j 
\end{equation}
in $F^{\sigma,2}_n$. 
}

\subsection{Identification of linear terms}
Now, we are in a position to consider the contribution of $F^{\sigma,3}_j$.
Recalling the definition of $F^{\sigma,3}_j$, degree-one terms, i.e.,
terms proportional to \Add{$\overline{r_{j+1}^3}$ and $\overline{r_j^3}$} is dominant,
so that the other term concerning $p_j\overline{r_{j+1}^2}$
is negligible (see \cref{lem:consequence_one_block_estimate}). 
In fact, we can show that the discrepancy between \Add{$\overline{r_j^3}$ and $3(c_2\beta)^{-1}\overline{r_j}$}
is negligible in view of \cref{lem:cubic_term_characterization}. 
Moreover, note that $r_{j+1}$ is replaced by $r_j$ with the help of summation-by-parts. 
Hence, the term governed by $F^{\sigma,3}_j$ can be rewritten as  
\begin{equation*}
\begin{aligned}
& \int_0^t \alpha n^{1/2}  \sum_{j\in\mathbb Z}
F^{\sigma,3}_j(s) \partial_x \varphi(\tfrac{j}{n}+v^\sigma_ns) ds \\
&\quad= \int_0^t \frac{\alpha}{n^{1/2}} \sum_{j\in\mathbb Z}
\Add{\frac{3\sigma}{c_2\beta} }\Big(- \frac{c_4}{6}  
+ \frac{c_3^2}{12 c_2} \Big) \overline{r_j}(s) 
\partial_x \varphi(\tfrac{j}{n}+v^\sigma_ns) ds + \mathcal R^{\sigma,n}_t(\varphi)\\
&\quad= 
\int_0^t \frac{\alpha}{n^{1/2}} \sum_{j\in\mathbb Z} 
\frac{\sigma}{\Add{4\beta} c_2^2} \Big(- 2 c_2 c_4
+ c_3^2 \Big) \overline{r_j}(s) 
\partial_x \varphi(\tfrac{j}{n}+v^\sigma_ns) ds + \mathcal R^{\sigma,n}_t(\varphi)
\end{aligned}
\end{equation*}
where the error term $\mathcal R^{\sigma,n}_t(\varphi)$ satisfies the bound
\begin{equation*}
\begin{aligned}
\mathbb E_n\bigg[\sup_{0\le t\le T}\Big| \mathcal R^{\sigma,n}_t(\varphi)\Big|^2 \bigg] 
\lesssim
\frac{C_T}{n}\left(\|\partial_x\varphi\|^2_{L^2(\mathbb R)}
+ \frac{T^2}{n^2} \|\partial_x^2 \varphi\|^2_{L^2(\mathbb R)}  \right).
\end{aligned}
\end{equation*}
Additionally, notice that 
\begin{equation*}
\frac{1}{2c_2^{3/2}}(\xi^+_j + \xi^-_j)
= \frac{1}{c_2} r_{j+1}. 
\end{equation*} 
Therefore, the computation for $F^{\sigma,3}_j$ and \Add{the linear drift \eqref{eq:linear_term_from_equipartition}} becomes 
\begin{equation*}
\begin{aligned}
&\int_0^t \alpha n^{1/2}\sum_{j\in\mathbb Z}
\Add{\Big( F^{\sigma,3}_j(s) 
+ \sigma \frac{c_3^2}{4c_2^2\beta} \varepsilon_n^2 \overline r_j(s) \Big)} 
\partial_x \varphi(\tfrac{j}{n}+v^\sigma_ns) ds \\
&= -\int_0^t
\sigma \alpha
\Add{\frac{c_2c_4 - c_3^2}{4c_2^{5/2} \beta}}
\Big( \mathcal X^{+,n}_s(\varphi_{v^\sigma_ns})
+  \mathcal X^{-,n}_s(\varphi_{v^\sigma_ns})\Big) ds + \mathcal R^{\sigma,n}_t(\varphi) .
\end{aligned}
\end{equation*}
Here, since one of the two modes $\xi^+_j$ and $\xi^-_j$ should be looked at in a wrong moving frame, due to the Riemann-Lebesgue lemma
(see \cref{lem:field_with_wrong_velocity}),
we have that in the last display survives in the limit 
only the mode $\sigma$ that we are concerned with.

Finally, let us combine the contribution from $F^{\sigma,0}_j$ and $II^\sigma$ - only those terms are left to bound, as a consequence of the preceding argument.   
Set 
\begin{equation*}
\tilde{II}^\sigma
= \alpha n^{1/2}  
\sum_{j\in\mathbb Z}
F^{\sigma,0}_j(\mathfrak r,\mathfrak p) \partial_x \varphi(\tfrac{j}{n}+v^\sigma_nt) .
\end{equation*}
Then, recalling the computation for $F^{\sigma,0}_j$, we have 
\begin{equation*}
\tilde{II}^\sigma
=\sigma c_2 \frac{\alpha}{n^{1/2}}
\sum_{j\in\mathbb Z} r_j \partial^2_x \varphi(\tfrac{j}{n}+v^\sigma_nt)  
+ O\Big(\frac{1}{n} \Big)
\end{equation*}
where we applied summation-by-parts and replaced the discrete derivative by the continuous one. 
Consequently, 
\begin{equation*}
\begin{aligned}
II^\sigma + \tilde{II}^\sigma
&=\frac{\alpha}{n^{1/2}}
\sum_{j\in\mathbb Z} \Big( \frac{1}{2}J^{\tilde{\xi}^\sigma}_{j,j+1} 
+\sigma c_2r_j \Big) 
\partial_x^2 \varphi(\tfrac{j}{n}+v^\sigma_nt)  
+ O(1/n) \\
&= \frac{\alpha}{n^{1/2}} \sigma\frac{\sqrt{c_2}}{2}
\sum_{j\in\mathbb Z} \big(\sqrt{c_2} r_j - \sigma p_j \big)  
\partial_x^2 \varphi(\tfrac{j}{n}+v^\sigma_nt)  
+ O(n^{-1/2} ) .
\end{aligned}
\end{equation*}
However, the first term in the utmost right-hand side of the last display is nothing but the fluctuation field of each phonon mode $\sigma$, which is looked at a wrong frame. 
Again, according to the Riemann-Lebesgue lemma (see \cref{lem:field_with_wrong_velocity}), we can see that this field does not survive in the limit. 

\subsection{Identification of viscosity terms} 
Here let us compute the action of the symmetric part of the generator to show \eqref{eq:martingale_decomposition}. 
Recalling the definition of the operator $S_n$, we can easily see that 
\begin{equation*}
S_n \tilde{\mathcal X}^{\sigma,n}_t(\varphi_{v^\sigma_nt})
= \frac{\gamma}{2n^{1/2}} \sum_{j\in\mathbb Z} 
\Big( 
\sigma p_j 
+ 
\mathfrak u \varepsilon_n \frac{p_j^2}{2}
\Big)
\Delta^n\varphi(\tfrac{j}{n}+v^\sigma_nt)
\end{equation*}
where $\Delta^n$ is the one-dimensional Laplacian defined in \eqref{eq:definition_discrete_derivative}. 
Note that 
\begin{equation}
\label{eq:symmetric_part_identity}
\begin{aligned}
2 p_j = \tilde{\xi}^+_j -  \tilde{\xi}^-_j 
\end{aligned}
\end{equation}
Then, we apply Lemma \ref{lem:field_with_wrong_velocity} to know that the field with a wrong frame does not affect the limit.
In particular, one of the first and second terms in the
utmost right-hand side of \eqref{eq:symmetric_part_identity} is negligible.
Moreover, as a consequence of a crude $L^2$-estimate, we have that 
\begin{equation*}
\mathbb{E}_n \bigg[ \sup_{0\le t\le T} \bigg|\int_0^t 
\frac{\varepsilon_n}{\sqrt{n}}\sum_{j\in\mathbb Z} 
\frac{p_j(s)^2}{2} \Delta^n\varphi(\tfrac{j}{n}+v^\sigma_ns) ds \bigg|^2 \bigg]
\lesssim \varepsilon_n^2 \| \partial_x^2 \varphi \|^2_{L^2(\mathbb R)} 
\end{equation*}
for any $v\in\mathbb R$ and thus the corresponding  term does not affect the limit.
Additionally, we can easily show that the discrepancy between discrete and continuous Laplacian is estimated by $O(n^{-2})$, and thus we conclude here that 
\begin{equation*}
\int_0^t S_n\tilde{\mathcal X}^{\sigma,n}_s(\varphi_{v^\sigma_ns})ds 
= \frac{\gamma}{4}\int_0^t 
\mathcal X^{\sigma,n}_s(\varphi_{v^\sigma_ns})ds 
+ O( \varepsilon_n). 
\end{equation*}
By this line, we complete the proof of \cref{lem:martingale_decomposition_main}. 


\section{The second-order Boltzmann-Gibbs principle}
\label{sec:2bg} 

To show the tightness and identify the limits in the decomposition \eqref{eq:martingale_decomposition}, we use the following replacement argument by local averages. 
For each $\ell\in\mathbb N$ and for any sequence of random variables $g=(g_j)_{j\in\mathbb Z}$, let 
\begin{equation*}
\overrightarrow g^\ell_j
= \frac{1}{\ell}\sum_{i=0}^{\ell-1}
\overline g_{j+i}
\end{equation*}
where recall that $\overline g_j=g_j-E_{\nu_n}[g_j]$ denotes the centered variable. 
Moreover, recall that we introduced the notation $\varphi(j, t) = \varphi(\tfrac{j}{n}+v^\sigma_nt)$. 

\begin{proposition}[The second-order Boltzmann-Gibbs principle]
\label{prop:2bg}
For any $\ell\in\mathbb N$ and for any $\varphi\in\mathcal S(\mathbb R)$, we have that 
\begin{equation}
\label{eq:2bg_estimate}
\begin{aligned}
& \mathbb E_n \bigg[ \sup_{0\le t\le T} \bigg|\int_0^t \sum_{j\in\mathbb Z}
\Big( p_j(s) \overline{r}_{j+1}(s)-
\overrightarrow{p}_{j}^\ell(s) \overrightarrow{r}_{j}^\ell(s)
\Big) \varphi(j,s) ds \bigg|^2 \bigg] \\
&\quad\lesssim T\Big(\frac{\ell}{n} + \frac{1}{\ell} \Big)  
\| \varphi\|^2_{L^2(\mathbb R)} + T^2 \frac{1}{n}\| \varphi'\|^2_{L^2(\mathbb R)}.
\end{aligned}
\end{equation}
\end{proposition}

Recall from \eqref{eq:antisymmetric_current_dominant} that we defined $W_j= - \frac{c_3}{2c_2^{3/2}} p_j\overline{r}_{j+1} \ven$ the dominant term
of the nonlinear contribution.
Since
$4 c_2 \overrightarrow{r}_{j}^\ell \overrightarrow{p}_{j}^\ell =
\left(\overrightarrow{(\xi^{+})}_{j}^\ell\right)^2 - \left(\overrightarrow{(\xi^{-})}_{j}^\ell\right)^2 $,
The above result enables us to replace $W_j$ with
\begin{equation*}
\begin{split}
- \frac{c_3}{2c_2^{3/2}} \overrightarrow{p}_{j}^\ell \overrightarrow{r}_{j}^\ell\ven
=  - \frac{c_3}{8c_2^{2}}
\left[\left(\overrightarrow{(\xi^{+})}_{j}^\ell\right)^2 - \left(\overrightarrow{(\xi^{-})}_{j}^\ell\right)^2\right]\ven.
\end{split}
\end{equation*}

The proof of the second-order Boltzmann-Gibbs principle we give here differs from
the one in \cite{gonccalves2014nonlinear} (see also \cite{gonccalves2017second}),
because of the degeneracy of the noise acting only on the momenta.
In order to establish it also for the positions we need to use the hamiltonian part of the dynamics. 
Here we decompose the discrepancy in the assertion \eqref{eq:2bg_estimate} as 
\begin{equation*}
\begin{aligned}
\overline{r}_{j+1} p_{j} 
- (\overrightarrow r^\ell_{j+1}) \overrightarrow p^\ell_{j}
=\overline{r}_{j+1} (p_{j} - \overrightarrow p^\ell_{j})
+ \overrightarrow p^\ell_{j} (\overline{r}_{j+1} - \overrightarrow r^\ell_{j+1}) . 
\end{aligned}
\end{equation*}
Notice that we can bound $(\overrightarrow r^\ell_{j+1}) 
\big( (\overrightarrow p^\ell_{j+1}) - \overrightarrow p^\ell_{j} \big)$ by paying a price of $\ell/n$ in the right-hand side of the assertion. 
Then, \cref{prop:2bg} follows by two lemmas treating separately the two terms.

\begin{lemma}[One-block estimate]
\label{lem:one_block_estimate}
For each $\ell\in\mathbb N$, we have that 
\begin{equation*}
\begin{aligned}
\mathbb E_n\bigg[\sup_{0\le t\le T}\bigg|\int_0^t \sum_{j\in\mathbb Z} \overline{r}_{j+1}
(p_{j} -\overrightarrow p^\ell_j)
\varphi(j,s) ds \bigg|^2 \bigg]
\lesssim \frac{T\ell}{n} \|\varphi\|^2_{L^2(\mathbb R)}. 
\end{aligned}
\end{equation*}
\end{lemma}

\begin{proof}
Note that 
\begin{equation*}
p_{j} -\overrightarrow p^\ell_j
= \sum_{i=0}^{\ell-2}(p_{j+i}-p_{j+i+1}) \psi_{i-1} 
\end{equation*}
where $\psi_i=(\ell-i)/\ell$. 
From this identity, we have 
\begin{equation*}
\begin{aligned}
\sum_{j\in\mathbb Z}\overline{r}_{j+1} (p_{j} -\overrightarrow p^\ell_j) \varphi(j,t) 
&= \sum_{j\in\mathbb Z}\overline{r}_{j+1}\sum_{i=0}^{\ell-2}(p_{j+i}-p_{j+i+1})\psi_{i-1} \varphi(j,t)
= \sum_{k\in\mathbb Z} F_k (p_k-p_{k+1})
\end{aligned}
\end{equation*}
where $F_j = \sum_{i=0,\ldots,\ell-2} \overline{r}_{j+1-i} \varphi(j-i,t) \psi_{i-1}$.
Above, in the second identity, we set $k=j-i$ to rearrange the sum. 
Here, we note that the local function $F_k$ does not depend on either $(r_k, p_k)$ or $(r_{k+1},p_{k+1})$.


Take an arbitrary local $L^2(\nu_n)$-function $f$ on $\mathscr X$ and
by Young's inequality, we have that 
\begin{equation*}
\begin{aligned}
2 \sum_{j\in\mathbb Z} 
E_{\nu_n} \big[F_j (p_j-p_{j+1}) f(\mathfrak r, \mathfrak p) \big] 
&=
2 \sum_{j\in\mathbb Z} 
E_{\nu_n}\big[ F_j p_{j+1} (f(\mathfrak r, \mathfrak p^{j,j+1})-f(\mathfrak r, \mathfrak p)) \big] \\
&\le \frac{1}{K} \sum_{j\in\mathbb Z} 
E_{\nu_n}\big[ F_j^2 p_{j+1}^2 \big]
+ K \sum_{j\in\mathbb Z} E_{\nu_n}\big[ (f(\mathfrak r, \mathfrak p^{j,j+1})-f(\mathfrak r, \mathfrak p))^2 \big]\\
&= \frac{1}{K\beta} \sum_{j\in\mathbb Z} E_{\nu_n}\big[ F_j^2 \big]
+ K \sum_{j\in\mathbb Z} E_{\nu_n}\big[ (f(\mathfrak r, \mathfrak p^{j,j+1})-f(\mathfrak r, \mathfrak p))^2 \big]
\end{aligned}
\end{equation*}
for any $K>0$. Now, since
\begin{equation*}
\sum_{j\in\mathbb Z} 
E_{\nu_n} \big[F_j^2\big]
= \sum_{j\in\mathbb Z}\sum_{i=0}^{\ell-1} E_{\nu_n}[\overline{r}_{j+1-i}^2] \varphi(j-i,t)^2 \psi_i^2
\lesssim \ell \beta^{-1} \sum_{j\in\mathbb Z} \varphi(j,t)^2 
\lesssim \ell n \beta^{-1} \|\varphi\|^2_{L^2(\mathbb R)},
\end{equation*}
choosing $K=O(n^2)$, this gives the bound
\begin{equation}
\label{eq:h-1p}
\bigg\| \sum_{j\in\mathbb Z}F_j (p_j-p_{j+1}) \bigg\|_{-1}^2
\lesssim \ell n \beta^{-2} \|\varphi\|^2_{L^2(\mathbb R)}.
\end{equation}
Hence, \cref{prop:kipnis_varadhan_estiamte} completes the proof.
\end{proof}

\begin{lemma}[Two-block estimate]
\label{lem:two_block_estimate}
For each $\ell\in\mathbb N $, we have that 
\begin{equation*}
\begin{aligned}
&\mathbb E_n\bigg[\sup_{0\le t \le T} \bigg|\int_0^t \sum_{j\in\mathbb Z}
\overrightarrow p^\ell_{j}(\overline{r}_{j+1} - \overrightarrow r^\ell_{j+1}) 
\varphi(j,s) ds \bigg|^2 \bigg]
\lesssim T\left(\frac{ \ell}{n} + \frac{1}{\ell}\right)\| \varphi\|^2_{L^2(\mathbb R)}
+ T^2 \frac{1}{n}\| \varphi'\|^2_{L^2(\mathbb R)}.
\end{aligned}
\end{equation*}
\end{lemma}

\begin{proof}
Notice that 
\begin{equation*}
\begin{aligned}
\overline{r}_{j+1}-\overrightarrow r^\ell_{j+1} 
= \sum_{i=0}^{\ell-2}(r_{j+1+i}-r_{j+i+2})\psi_{i-1} 
\end{aligned}
\end{equation*}
where $\psi_i=(\ell-i)/\ell$. 
Then, we compute 
\begin{equation*}
\begin{aligned}
&\bigg\langle 2\sum_{j\in\mathbb Z} \overrightarrow p^\ell_{j}
(\overline{r}_{j+1}-\overrightarrow r^\ell_{j+1}) \varphi(j,t) , f(\mathfrak r, \mathfrak p) \bigg\rangle \\
&\quad=2\sum_{j\in\mathbb Z} \sum_{i=0}^{\ell-2}E_{\nu_n}
\big[  \overrightarrow p^\ell_{j} (r_{j+i+1}-r_{j+i+2})\psi_{i-1} \varphi(j,t) f(\mathfrak r, \mathfrak p) \big]. 
\end{aligned}
\end{equation*}
Here, observe that
\begin{equation*}
\begin{split}
L\left( \overrightarrow p^\ell_{j} p_k \right) &=
\alpha \left( \overrightarrow p^\ell_{j}  A p_k + p_k A \overrightarrow p^\ell_{j} \right)
+ \gamma S(\overrightarrow p^\ell_{j} p_k)\\
&= \alpha \left(\overrightarrow p^\ell_{j} \left(V_n'(r_{k+1}) -  V_n'(r_{k})\right)
+ p_k A \overrightarrow p^\ell_{j} \right)
+ \gamma S(\overrightarrow p^\ell_{j} p_k)\\
&= \alpha \Big( \overrightarrow p^\ell_{j} c_2 \left(r_{k+1} -  r_{k}\right)
+ p_k A \overrightarrow p^\ell_{j} \Big)
+ \alpha \overrightarrow p^\ell_{j}\ve_n c_3 \left(\overline{r_{k+1}^2} - \overline{ r_{k}^2}\right) +
\alpha \overrightarrow p^\ell_{j} O(\ve_n^2)\\
&\quad+ \gamma S(\overrightarrow p^\ell_{j} p_k). 
\end{split}
\end{equation*}
This gives
\begin{equation*}
\begin{split}
&\sum_{j\in\mathbb Z} \sum_{i=0}^{\ell-2}
\overrightarrow p^\ell_{j} (r_{j+i+1}-r_{j+i+2})\psi_{i-1} \varphi(j,t)\\
&\quad= -\frac 1{\alpha c_2}  \sum_{j\in\mathbb Z} \sum_{i=0}^{\ell-2}
L\left(\overrightarrow p^\ell_{j} p_{j+i+1} \psi_{i-1} \varphi(j,t)\right) 
+ \frac{\gamma}{\alpha c_2}\sum_{j\in\mathbb Z} \sum_{i=0}^{\ell-2} S(\overrightarrow p^\ell_{j} p_{j+i+1})
\psi_{i-1} \varphi(j,t)\\
&\qquad+ \sum_{j\in\mathbb Z} \sum_{i=0}^{\ell-2} p_{j+i+1} A \overrightarrow p^\ell_{j} \psi_{i-1} \varphi(j,t)
+ \frac{c_3}{c_2} \sum_{j\in\mathbb Z} \sum_{i=0}^{\ell-2}
\overrightarrow p^\ell_{j}\ve_n \left(\overline{r_{j+i+2}^2} - \overline{ r_{j+i+1}^2}\right)  \psi_{i-1} \varphi(j,t)
\end{split}
\end{equation*}
plus some smaller terms. 
Let us handle each term in the right-hand side of the last display. 
First, by Dynkin's martingale formula, the fluctuating term is written as 
\begin{equation*}
\begin{split}
& \int_0^t (n^{-2} \partial_s + L)\left( \sum_{j\in\mathbb Z} \sum_{i=0}^{\ell-2}
\overrightarrow p^\ell_{j}(s) p_{j+i+1}(s) \psi_{i-1} \varphi(j,s) \right) \; ds \\
&\quad= \frac {1}{n^2}
\sum_{j\in\mathbb Z}\sum_{i=0}^{\ell-2}\left[   \overrightarrow p^\ell_{j}(t) p_{j+i+1}(t) \psi_{i-1} \varphi(j,t) -
\overrightarrow p^\ell_{j}(0) p_{j+i+1}(0) \psi_{i-1} \varphi(j,0) \right]
\end{split}
\end{equation*}
with some martingale, which is negligible as we did in the proof of \cref{lem:equipartition_energy}.  
Moreover, we can bound the square of first term of the last display by
\begin{equation*}
\begin{split}
\frac {1}{n^4}
E_{\nu_n} \left[\bigg(
\sum_{j\in\mathbb Z}\sum_{i=0}^{\ell-2}  \overrightarrow p^\ell_{j} p_{j+i+1} \psi_{i-1} \varphi_j \bigg)^2\right]
\le C \frac{\ell^2}{n^3},
\end{split}
\end{equation*}
which is clearly negligible.
Next, regarding the error term due to the nonlinearity, we have the bound 
\begin{equation}
\label{eq:errve}
\begin{split}
&\frac 1n E_{\nu_n}\left[\bigg( \sum_{j\in\mathbb Z}
\overrightarrow p^\ell_{j} \Big( \overline{r_{j+1}^2} -  \overrightarrow{(\overline{r}^2)}^\ell_{j+1}\bigg)   \varphi_j\Big)^2\right] \\
&\quad\le \frac 1n E_{\nu_n}\left[\bigg( \sum_{j\in\mathbb Z}
\overrightarrow p^\ell_{j} \overline{r_{j+1}^2} \varphi_j\bigg)^2 \right]
+ \frac 1n E_{\nu_n}\left[\left(\sum_{j\in\mathbb Z}
\overrightarrow p^\ell_{j}  \overrightarrow{(\overline{r}^2)}^\ell_{j+1}  \varphi_j\right)^2\right] \\
&\quad\le \frac Cn \sum_{j\in\mathbb Z} E_{\nu_n}\left[( \overrightarrow p^\ell_{j})^2\right] \varphi_j^2 +
\frac 1n \sum_{|j-j'|\le \ell} E_{\nu_n}\left[ \overrightarrow p^\ell_{j}\overrightarrow p^\ell_{j'}\right] E_{\nu_n}\left[\overrightarrow{(\overline{r}^2)}^\ell_{j+1}    \overrightarrow{(\overline{r}^2)}^\ell_{j'+1}\right]
\varphi_j \varphi_{j'} \\
&\quad\le \frac C\ell \|\varphi\|^2_{L^2}.
\end{split}
\end{equation}
Additionally, we have that
\begin{equation}
\label{eq:Sp}
\Big\| S \sum_{j\in\mathbb Z} \sum_{i=0}^{\ell-2} \overrightarrow p^\ell_{j} p_{j+i+1} 
\psi_{i-1} \varphi_j\Big\|_{-1}^2 =
\Big\| \sum_{j\in\mathbb Z} \sum_{i=0}^{\ell-2} \overrightarrow p^\ell_{j} p_{j+i+1}
\psi_{i-1} \varphi_j\Big\|_{1}^2,
\end{equation}
which can be bounded similarly to \eqref{eq:h-1p}.

For the time derivative term we have
\begin{equation*}
\begin{split}
\frac 1{n^2} \sum_{j\in\mathbb Z} \sum_{i=0}^{\ell-2}
\overrightarrow p^\ell_{j}p_{j+i+1}  \psi_{i-1} \partial_t \varphi(j,s)
=  \frac {v^\sigma_n}{n^2} \sum_{j\in\mathbb Z} \sum_{i=0}^{\ell-2}
\overrightarrow p^\ell_{j}p_{j+i+1}  \psi_{i-1} (\varphi')(j,s). 
\end{split}
\end{equation*}
Then, 
\begin{equation*}
\begin{split}
\mathbb E\left(\left[ \sup_{0<t<T} \int_0^t \frac 1{n^2} \sum_{j\in\mathbb Z} \sum_{i=0}^{\ell-2}
\overrightarrow p^\ell_{j}p_{j+i+1}  \psi_{i-1} \partial_s \varphi(j,s) ds \right]^2\right)
\le \frac{C T^2 \ell^2}{n} \|  \varphi'\|_{L^2}^2
\end{split}
\end{equation*} 
Now, we are left with $\sum_{j\in\mathbb Z} \sum_{i=0}^{\ell-2} p_{j+i+1} A \overrightarrow p^\ell_{j} \psi_{i-1} \varphi_j$.
Setting $F^{(p)}_j = \sum_{i=0}^{\ell-2}\psi_{i-1} p_{j+i+1}$, we have that 
\begin{equation*}
\begin{split}
\sum_{j\in\mathbb Z} \sum_{i=0}^{\ell-2} p_{j+i+1} A \overrightarrow p^\ell_{j} \psi_{i-1} \varphi_j
&= \sum_{j\in\mathbb Z} F^{(p)}_j \frac{\overline{V_n'(r_{j+\ell})} - \overline{V_n'(r_j)}}{\ell}  \varphi_j\\
&=\frac 1\ell \sum_{j\in\mathbb Z} (F^{(p)}_{j-\ell} -  F^{(p)}_j) \overline{V_n'(r_j)}  \varphi_{j-\ell}
+ \frac 1\ell \sum_{j\in\mathbb Z} (\varphi_{j-\ell} - \varphi_j)  F^{(p)}_j \overline{V'(r_j)} . 
\end{split}
\end{equation*}
Note that the expectation of the square of the second term in the utmost right-hand side of the last display is bounded by
\begin{equation*}
E_{\nu_n}\left[ \left(\frac 1\ell \sum_{j\in\mathbb Z} (\varphi_{j-\ell} - \varphi_j)  F^{(p)}_j \overline{V'(r_j)}\right)^2\right]
\le \frac {C}{\ell^2} \sum_{j\in\mathbb Z}(\varphi_{j-\ell} - \varphi_j) ^2 \le \frac{C}{n} \|\varphi'\|^2_{L^2}.
\end{equation*}
For the other term we can estimate the $H_{-1}$ norm as follows: 
\begin{equation*}
\begin{split}
&E_{\nu_n}\left[\frac 1\ell \sum_{j\in\mathbb Z} (F^{(p)}_{j-\ell} -  F^{(p)}_j) \overline{V'(r_j)}
\varphi_{j-\ell} f(\mathfrak r, \mathfrak p) \right]\\
& = \frac 1\ell \sum_{j\in\mathbb Z}\sum_{i=0}^{\ell-2}\psi_{i-1}
E_{\nu_n}\left[ (p_{j+i+1-\ell}  -  p_{j+i+1}) \overline{V'(r_j)} f(\mathfrak r, \mathfrak p) \right]
\varphi_{j-\ell}\\
& = \frac 1\ell \sum_{j\in\mathbb Z}\sum_{i,i'=0}^{\ell-2}\psi_{i-1}
E_{\nu_n}\left[ (p_{j+i-i'}  -  p_{j+i- i'+1}) \overline{V'(r_j)} f(\mathfrak r, \mathfrak p) \right]
\varphi_{j-\ell}\\
& = \frac 1\ell \sum_{j\in\mathbb Z}\sum_{i,i'=0}^{\ell-2}\psi_{i-1}
E_{\nu_n}\left[ (p_{j}  -  p_{j+1}) \overline{V'(r_{j-i+i'})} f(\mathfrak r, \mathfrak p) \right]
\varphi_{j-i+i'}\\
& = \frac 1\ell \sum_{j\in\mathbb Z}\sum_{i,i'=0}^{\ell-2}\psi_{i-1}
E_{\nu_n}\left[ p_{j} \overline{V'(r_{j-i+i'})}
\left(f(\mathfrak r, \mathfrak p) - f(\mathfrak r, \mathfrak p^{j,j+1})\right)\right]
\varphi_{j-i+i'}. 
\end{split}
\end{equation*}
However, by Young's inequality, 
\begin{equation*}
\begin{split}
& E_{\nu_n}\left[\frac 1\ell \sum_{j\in\mathbb Z} (F^{(p)}_{j-\ell} -  F^{(p)}_j) \overline{V'(r_j)}
\varphi_{j-\ell} f(\mathfrak r, \mathfrak p) \right]\\
&\quad\le \frac{1}{2K} \sum_{j\in\mathbb Z}
E_{\nu_n}\left[\left( \frac 1\ell \sum_{i,i'=0}^{\ell-2}\psi_{i-1} p_{j} \overline{V'(r_{j-i+i'})} \varphi_{j-i+i'}\right)^2\right] \\
&\qquad+ \frac{K}{2} \sum_{j\in\mathbb Z}
E_{\nu_n}\left[  \left(f(\mathfrak r, \mathfrak p) - f(\mathfrak r, \mathfrak p^{j,j+1})\right)^2\right]\\
& \quad\le \frac {Cn}{2K} \|\varphi\|^2_{L^2} 
+ \frac {K}{2} \sum_{j\in\mathbb Z}
E_{\nu_n}\left[  \left(f(\mathfrak r, \mathfrak p) - f(\mathfrak r, \mathfrak p^{j,j+1})\right)^2\right],
\end{split}
\end{equation*}
for any $K>0$. Thus, choosing $K=O(n^2)$, we have that 
\begin{equation*}
\left\| \frac 1\ell \sum_{j\in\mathbb Z} (F^{(p)}_{j-\ell} -  F^{(p)}_j) \overline{V'(r_j)}
\varphi_{j-\ell} \right\|_{-1}^2 \le C n.  
\end{equation*}
Hence, \cref{prop:kipnis_varadhan_estiamte} completes the proof.

\end{proof}

\section{Riemann-Lebesgue Estimates}
\label{sec:riemann_lebesgue_estimates}

\subsection{Handling fields with wrong velocity} 
Here we present some estimates involved with a Riemann-Lebesgue lemma. 
First we show that the fluctuation fields, when they are looked at in a divergent velocity in a wrong direction, do not survive in the limit.  
This is essentially established from the viewpoint of the result shown in~\cite[Theorem 5.1]{cannizzaro2024abc}, albeit they study a particle system on the torus.  
Instead of giving a quite general form of the statement as is done in \cite{cannizzaro2024abc}, here let us give more direct proof, in order to kill fluctuation fields with divergent moving frames to a wrong direction. 
Let us begin from handling the fields with wrong velocity.   

\begin{lemma}
\label{lem:field_with_wrong_velocity}
Fix $\sigma=\pm$ and let $\{\tilde{v}_n\}_{n\in\mathbb N}$ be a real sequence such that $\lim_{n\to\infty}|v^\sigma_n-\Tilde{v}_n|=\infty$. 
Then, we have that 
\begin{equation*}
\limsup_{n\to\infty}
\mathbb E_n \bigg[
\sup_{0\le t\le T} \bigg| \int_0^t 
\widetilde{\mathcal{X}}^{\sigma,n}_s(\varphi _{\tilde{v}_ns}) ds \bigg| \bigg]
=0 .
\end{equation*}
\end{lemma}

\begin{proof}
\Add{ 
Let us introduce the Friedrich mollifier $\psi\in\mathcal S(\mathbb R)$
and denote $\psi^\tau(x)=\tau^{-1} \psi(\tau^{-1}x)$ for any $\tau>0$ .
Let us decompose $\varphi=\varphi\ast \psi^\tau+(\varphi - \varphi\ast \psi^\tau) $ where $\ast$ denotes the convolution. 
Then, note that 
\begin{equation}
\label{eq:field_decomposition_with_mollifier1}
\begin{aligned}
\widetilde{\mathcal{X}}^{\sigma,n}_t((\varphi\ast \psi^\tau)_{\tilde{v}_nt})
&=\int_{\mathbb R} \frac{1}{\sqrt{n}}\sum_{j\in\mathbb Z} \overline\xi^\sigma_j(t) 
\varphi(x)\psi^\tau(\tfrac{j}{n}+\widetilde{v}_nt -x) dx \\
&= \int_{\mathbb R} \frac{1}{\sqrt{n}}\sum_{j\in\mathbb Z} \overline\xi^\sigma_j(t) 
\varphi(x+(\tilde{v}_n-v_n)t) \psi^\tau(\tfrac{j}{n}+v_nt -x) dx \\
&= \int_{\mathbb R} \widetilde{\mathcal X}^{\sigma,n}_t(\psi^\tau_{v^\sigma_nt-x}) \varphi(x+(\tilde{v}_n-v^\sigma_n)t) dx. 
\end{aligned}
\end{equation}
Here note that the fluctuation field with the correct velocity satisfies  
\begin{equation}
\label{eq:fluctuation_field_hoelder}
\begin{aligned}
\limsup_{n\to\infty} 
\mathbb E_n 
\Big[ \big| \widetilde{\mathcal{X}}^{\sigma,n}_t(\varphi_{v^\sigma_nt})
- \widetilde{\mathcal{X}}^{\sigma,n}_s(\varphi_{v^\sigma_ns}) \big|\Big] 
\le C(T,\varphi) |t-s|^{1/2} 
\end{aligned}
\end{equation}
for any $s,t\in [0,T]$ and for any $\varphi\in\mathcal S(\mathbb R)$, with some $C(T,\varphi)>0$.   
The proof of the last estimate \eqref{eq:fluctuation_field_hoelder} is based on the martingale decomposition and it follows by estimating each term separately which is done in \cref{sec:tightness}.  
This H\"older estimates yields that for any $x\in \mathbb R$
\begin{equation}
\label{eq:riemann_lebesgues_neglecting_riemann_sum}
\begin{split}
\limsup_{n\to\infty} \mathbb E_n\bigg[\sup_{0\le t\le T}
\left|\int_0^t \int_{\mathbb R} \widetilde{\mathcal{X}}^{\sigma,n}_s(\psi^\tau_{ {v}^\sigma_ns-x})
\varphi(x + (\tilde{v}_n - v_n^{\sigma}) s)  dx ds  \right|\bigg]
= 0,
\end{split}
\end{equation}
since 
\begin{equation*}
\begin{aligned}
& \mathbb E\bigg[\sup_{0\le t\le T}
\left|\int_0^t \int_{\mathbb R} \widetilde{\mathcal{X}}^{\sigma,n}_s(\psi^\tau_{ {v}^\sigma_ns-x})
\varphi(x + (\tilde{v}_n - v_n^{\sigma}) s)  dx ds  \right|\bigg]\\
&\quad\le \int_0^T 
\int_{\mathbb R} 
\mathbb E_n\Big[ \big| 
\widetilde{\mathcal{X}}^{\sigma,n}_t(\psi^\tau_{ {v}^\sigma_nt-x})\big| \Big] \varphi(x+(\tilde{v}_n-v_n^\sigma)s)dx ds \\
&\qquad+ \int_0^T 
\int_{\mathbb R} 
\mathbb E_n\Big[ \big| 
\widetilde{\mathcal{X}}^{\sigma,n}_t(\psi^\tau_{ {v}^\sigma_nt-x})
- \widetilde{\mathcal{X}}^{\sigma,n}_s(\psi^\tau_{ {v}^\sigma_ns-x})\big| \Big] \varphi(x+(\tilde{v}_n-v_n^\sigma)s)dx ds 
\end{aligned}
\end{equation*}
and it is not hard to see that both terms in the right-hand side of the last display vanish as $n\to\infty$. 
}
\Add{
On the other hand, by a crude estimate we have
\begin{equation*}
\limsup_{n\to\infty} 
\mathbb E_n\bigg[ \sup_{0\le t\le T} 
\bigg|\int_0^t  
\widetilde{\mathcal{X}}^{\sigma,n}_s((\varphi-\varphi\ast \psi^\tau)_{\tilde{v}_ns}) 
ds \bigg| \bigg] 
\le C(T) \| \varphi-\varphi\ast\psi_\tau\|_{L^2(\mathbb R)} \end{equation*}
with some universal constant $C(T)>0$ which is independent of $\tau$. 
Then, by definition of the mollifier, the right-hand side of the last display vanishes as $\tau \to 0$, and thus we complete the proof. 
}
\end{proof}

Next, we show an analogous result for the quadratic fields with wrong velocity. 

\begin{lemma}
\label{lem:quadratic_field_with_wrong_velocity}
Given $\sigma=\pm$, let $\{\tilde{v}_n\}_{n\in\mathbb N}$ be a real sequence such that
$\lim_{n\to\infty}|\tilde{v}_n-v^\sigma_n|=\infty$.
Then, we have that 
\begin{equation*}
\limsup_{n\to\infty} 
\mathbb E_n \bigg[\sup_{0\le t\le T}\bigg| \int_0^t 
\sum_{j\in\mathbb Z}
\overline{\xi}^\sigma_j(s)\overline{\xi}^\sigma_{j+1}(s) \varphi(\tfrac{j}{n}+\tilde{v}_ns ) ds  \bigg| \bigg]
=0.
\end{equation*}
\end{lemma}

\begin{proof}
\Add{ 
First, we note that by the second-order Boltzmann-Gibbs principle (\cref{prop:2bg}),
we can replace $\overline{\xi}^\sigma_j\overline\xi^\sigma_{j+1}$ by the square of local averages:
for any $\delta>0$
\begin{equation}
\label{eq:6}
\begin{split}
\lim_{n\to\infty} \mathbb E_n\bigg[ \sup_{0\le t\le T} 
\bigg| 
\int_0^t \sum_{j\in\mathbb Z} 
\Big(\overline\xi^\sigma_j(s)\overline\xi^\sigma_{j+1}(s) 
- \big((\overrightarrow{\xi^\sigma})^{[\delta n]}_j(s) \big)^2 \Big)
\varphi(\tfrac{j}{n}+\tilde{v}_ns )  ds \bigg|^2 \bigg]
\le \delta\; CT \| \varphi\|^2_{L^2(\mathbb R)}.
\end{split}
\end{equation}
For the local average, we can easily see that we have an identity 
\begin{equation*}
\mathcal X^{\sigma,n}_t\big(T^-_{j/n}\iota_\delta \big) 
= \frac{1}{\sqrt{n}}
\sum_{j'\in\mathbb Z} 
\overline{\xi}^\sigma_{j'}(t) 
\iota_\delta(\tfrac{j'-j}{n}) 
= \sqrt{n} (\overrightarrow{\xi^\sigma})^{[\delta n]}_j
\end{equation*}
where 
$\iota_\delta(\cdot) = \delta^{-1} \mathbf{1}_{[0,\delta)}(\cdot)$ and $T^\pm_{a}G(\cdot) = G(\cdot\pm a)$ for any $a\in\mathbb R$ and $G\in C(\mathbb R)$. 
Then we only have to show that
\begin{equation}
\label{eq:local_average_term}
\lim_{\delta\to 0}
\limsup_{n\to\infty} \mathbb E_n\bigg[ \sup_{0\le t\le T} 
\bigg| \int_0^t \frac{1}{n} \sum_{j\in\mathbb Z} 
\mathcal X^{\sigma,n}_s \big(T^-_{j/n}\iota_\delta\big)^2 \varphi(\tfrac{j}{n}+\tilde{v}_ns ) ds
\bigg|\bigg] = 0. 
\end{equation}
Analogously to the proof of~\cref{lem:field_with_wrong_velocity},
let $\psi$ be the Friedrich mollifier and set $\iota_\delta^\tau=\iota_\delta\ast\psi_\tau$
where $\psi_\tau(x)=\tau^{-1}\psi(\tau^{-1}x)$ for each $\tau>0$ and decompose 
\begin{equation}
\label{eq:local_average_term_decomposition}
\begin{aligned}
\mathcal X^{\sigma,n}_s \big(T^-_{j/n}\iota_\delta\big)^2
=& \mathcal X^{\sigma,n}_s \big(T^-_{j/n}\iota^\tau_\delta\big)^2
+ \big[ \mathcal X^{\sigma,n}_s \big(T^-_{j/n} \iota_\delta\big)^2-
\mathcal X^{\sigma,n}_s \big(T^-_{j/n}\iota^\tau_\delta\big)^2 \big]. 
\end{aligned}
\end{equation}
Regarding the second term in the right-hand side of the last display, note that  
\begin{equation*}
\begin{aligned}
& \mathbb E_n\bigg[\sup_{0\le t\le T}
\bigg|\int_0^t \frac{1}{n} \sum_{j\in\mathbb Z} 
\big[ \mathcal X^{\sigma,n}_s \big(T^-_{j/n}\iota_\delta\big)^2-\mathcal X^{\sigma,n}_s \big(T^-_{j/n}\iota^\tau_\delta\big)^2\big]
\varphi(\tfrac{j}{n}+\tilde{v}_ns ) ds
\bigg| \bigg]\\
&\le \int_0^T \frac{1}{n}\sum_{j\in\mathbb Z} 
\mathbb E_n\Big[ \big| \mathcal X^{\sigma,n}_s \big(T^-_{j/n}\iota_\delta\big)^2-\mathcal X^{\sigma,n}_s \big(T^-_{j/n}\iota^\tau_\delta\big)^2\big|\Big]  \big| \varphi(\tfrac{j}{n}+\tilde{v}_ns ) \big| ds\\
&\le \int_0^T \frac{1}{n}\sum_{j\in\mathbb Z} 
\mathbb E_n\Big[ \big| \mathcal X^{\sigma,n}_s \big(T^-_{j/n}\iota_\delta\big)
-\mathcal X^{\sigma,n}_s \big(T^-_{j/n}\iota^\tau_\delta\big)\big|
\big| \mathcal X^{\sigma,n}_s \big(T^-_{j/n}\iota_\delta\big)
+\mathcal X^{\sigma,n}_s \big(T^-_{j/n}\iota^\tau_\delta\big)\big|\Big]  \big| \varphi(\tfrac{j}{n}+\tilde{v}_ns ) \big| ds
\\
&\le \int_0^T \Big(\frac{1}{n}\sum_{j\in\mathbb Z} 
\mathbb E_n\Big[  \Big|\mathcal X^{\sigma,n}_s \big(T^-_{j/n}(\iota_\delta -\iota^\tau_\delta)\big)\Big|^2\Big]^{1/2}
\mathbb E_n\Big[\Big|\mathcal X^{\sigma,n}_s \big(T^-_{j/n}(\iota_\delta + \iota^\tau_\delta)\big)\Big|^2 \Big]^{1/2}
\big| \varphi(\tfrac{j}{n}+\tilde{v}_ns )\big| \Big) ds
\\
&\le CT \| \iota_\delta-\iota_\delta^\tau\|_{L^2(\mathbb R)}
\| \iota_\delta\|_{L^2(\mathbb R)}
\| \varphi\|_{L^1(\mathbb R)}
\end{aligned}
\end{equation*}
for some $C>0$, which is independent of $n$.  
In particular, note that the utmost right-hand side of the last estimate vanishes as $\tau\to0$ for any $\delta> 0$.
}
\Add{ 
Thus, our task is now to show that 
\begin{equation*}
\lim_{\delta\to 0}\limsup_{n\to\infty} \mathbb E_n\bigg[ \sup_{0\le t\le T} 
\bigg| \int_0^t 
\int_{\mathbb R} 
\widetilde{\mathcal X}^{\sigma,n}_s \big(T^-_{x}\iota^\tau_\delta\big)^2 \varphi(x+\tilde{v}_ns ) dx ds \bigg|\bigg] = 0
\end{equation*}
where we replaced the field $\mathcal X^{n,\sigma}$ by $\widetilde{\mathcal X}^{n,\sigma}$ by a crude $L^2$-estimate. 
To this end, by a change of variable, it is enough to show for any $\tau>0$,   
\begin{equation*}
\lim_{\delta \to 0} 
\limsup_{n\to\infty}
\sup_{x\in\mathbb R} 
\mathbb E_n\bigg[
\sup_{0\le t\le T}
\bigg| \int_0^t \widetilde{\mathcal X}^{\sigma,n}_s(T^+_{v^\sigma_ns-x} \iota^\tau_\delta)^2 
\varphi(x+(\tilde{v}_n-v_n^\sigma)s)  ds \bigg|\bigg]
= 0 
\end{equation*}
where we used Fubini's theorem and the dominated convergence theorem. 
However, the proof of the last assertion can be done in a similar way as \eqref{eq:riemann_lebesgues_neglecting_riemann_sum}, noting that we have the following H\"older estimate: for any $\varphi\in\mathcal S(\mathbb R)$,  
\begin{equation*}
\begin{aligned}
\mathbb E_n \Big[ \big| 
\widetilde{\mathcal X}^{\sigma,n}_t (\varphi_{v^\sigma_nt})
- \widetilde{\mathcal X}^{\sigma,n}_s
(\varphi_{v^\sigma_ns}) \big|^2\Big]
\le C(T,\varphi) |t-s|
\end{aligned}
\end{equation*}
with some $C(T,\varphi)>0$.
This estimate follows by the martingale decomposition and by handling each term separately, and the proof is analogous to \eqref{eq:fluctuation_field_hoelder}.  
Hence we complete the proof of the assertion. 
}
\end{proof}

Finally, note that we can show the following result for cross-fields. 

\begin{lemma}
\label{lem:quadratic_field_with_wrong_velocity_cross_fields}
Let $\{\tilde{v}_n\}_{n\in\mathbb N}$ be a real sequence such that
$\lim_{n\to\infty}|\tilde{v}_n-v^\sigma_n|=\infty$ for each $\sigma=\pm$ .
Then, we have that 
\begin{equation*}
\limsup_{n\to\infty} 
\mathbb E_n \bigg[\sup_{0\le t\le T}\bigg| \int_0^t 
\sum_{j\in\mathbb Z}
\overline{\xi}^-_j(s)\overline{\xi}^+_{j+1}(s) \varphi _{\tilde{v}_ns} (\tfrac{j}{n}) ds  \bigg| \bigg] =0. 
\end{equation*}
\end{lemma}

This result holds according to the fact that the quadratic field of $\overline{\xi}^-_j\overline{\xi}^+_{j+1}$
can be rewritten as a product of fluctuation fields of two different phonon modes,
one of which should be looked at in a moving frame
to a wrong direction.  
Moreover, note that the same assertion clearly holds when we replace
$\overline{\xi}^-_j\overline{\xi}^+_{j+1}$ by $\overline{\xi}^+_j\overline{\xi}^-_{j+1}$.

\subsection{Characterization of Martingale Terms} 
Recall \eqref{eq:carre_du_champs_computation}. 
Next, let us compute the covariance of the martingales $\mathcal{M}^{+,n}$ and $\mathcal{M}^{-,n}$. 
The computation of this part implies that the two white noises in the limiting equation are independent. 
For each $\overrightarrow{\varphi}=(\varphi^+ ,\varphi^-)\in\mathcal{S}(\mathbb R)^2$, we set 
\begin{equation*}
\mathcal{Z}^n_t (\overrightarrow \varphi) 
= \mathcal{X}^{+,n}_t (\varphi^+ _{v^+_nt})
+ \mathcal{X}^{-,n}_t (\varphi^-_{v^-_nt}).
\end{equation*}
By Dynkin's martingale formula, 
\begin{equation*}
\mathcal N^n_t (\overrightarrow{\varphi})
= \mathcal{Z}^n_t(\overrightarrow{\varphi})
- \mathcal{Z}^n_0(\overrightarrow{\varphi})
- \int_0^t (\partial_s + L_n) \mathcal{Z}^n_s (\overrightarrow{\varphi})ds 
\end{equation*}
and $\mathcal N^n_t(\overrightarrow{\varphi})^2- \langle \mathcal{N}^n(\overrightarrow{\varphi})\rangle_t $ are martingales where 
\begin{equation*}
\langle\mathcal{N}^n(\overrightarrow{\varphi})\rangle_t
= \int_0^t \big( L_n \mathcal{Z}^n_s(\overrightarrow{\varphi})^2 
- 2 \mathcal{Z}^n_s(\overrightarrow{\varphi}) L_n \mathcal{Z}^n_s(\overrightarrow{\varphi}) \big) ds .
\end{equation*}
Note that the anti-symmetric part does not contribute to the martingale part. 
This is based on the fact that the operator $A$ is a linear combination of continuous first-order differential operators. 
Then, we compute 
\begin{equation*}
\begin{aligned}
&L_n \mathcal{Z}^n_t(\overrightarrow{\varphi})^2 - 2 \mathcal{Z}^n_t(\overrightarrow{\varphi}) L_n \mathcal{Z}^n_t(\overrightarrow{\varphi})\\ 
&= \sum_{\sigma = \pm} \big( L_n\mathcal{X}^{\sigma,n}_t(\varphi^\sigma _{v^\sigma_nt})^2 
- 2 \mathcal{X}^{\sigma,n}_t(\varphi^\sigma _{v^\sigma_nt}) L_n \mathcal{X}^{\sigma,n}_t (\varphi^1 _{v^\sigma_nt}) \big)\\
&\quad + 2L_n\big( \mathcal{X}^{+,n}_t(\varphi^+ _{v^+_nt}) \mathcal{X}^{-,n}_t(\varphi _{v^-_nt}^-) \big)
- 2 \mathcal{X}^{+,n}_t(\varphi^+ _{v^+_nt}) L_n \mathcal{X}^{-,n}_t (\varphi _{v^-_nt}^-)
- 2 \mathcal{X}^{-,n}_t(\varphi _{v^-_nt}^-) L_n \mathcal{X}^{+,n}_t (\varphi _{v^+_nt}^+)  \\
&= \frac{1}{2} \sum_{\sigma = \pm} \sum_{j\in\mathbb Z} 
(\xi^{\sigma}_j - \xi^\sigma_{j+1})^2 (\nabla^n \varphi^\sigma _{v^\sigma_nt}(\tfrac{j}{n}))^2 \\
&\quad+ \frac{1}{n} \sum_{j\in\mathbb Z} 
(\xi^{+}_j - \xi^+_{j+1})(\xi^-_j -\xi^-_{j+1}) 
\big(\nabla^n \varphi^+ _{v^+_nt}(\tfrac{j}{n})\big)
\big(\nabla^n \varphi _{v_n^-t}^-(\tfrac{j}{n})\big).
\end{aligned}
\end{equation*}
Then, we can easily see that as $n\to\infty$, the first term
in the utmost right-hand side of the last display converges to
$2\beta^{-1}\sum_{\sigma= \pm} \| \partial_x \varphi^\sigma \|^2_{L^2(\mathbb R)}$.
Above, note that $\mathrm{Var}_{\nu_n}[\xi^\pm_0]=2\beta^{-1}$. 
On the other hand, the second term of the last display vanishes as $n\to\infty$ due to the following result. 

\begin{lemma}
Let $\{a_n\}_{n\in\mathbb N}$ be a real sequence such that $\lim_{n\to\infty}a_n =\infty$. 
Then we have that 
\begin{equation*}
\limsup_{n\to\infty} 
\bigg| \frac{1}{n}\sum_{j\in\mathbb Z}
\varphi^{1} \Big(\frac{j}{n}\Big) \varphi^2 \Big(\frac{j}{n}+a_n\Big) \bigg|=0 .
\end{equation*}
\end{lemma}
\begin{proof}
We expand the test functions as a series of exponential functions as 
\begin{equation*}
\varphi^k (x) = 
\int_{\mathbb R} \widehat \varphi^k(\xi) e^{2\pi \mathsf i \xi x} d\xi  
\end{equation*}
where $\widehat \varphi^k(\xi)$
is the Fourier transform of $\varphi^k$ for each $k=1,2$. 
Note that we have the relation 
$\sum_{j\in\mathbb Z} e^{2\pi\mathsf i j \xi} = \delta(\xi)$
where $\delta$ is the delta measure. 
Then, we have that
\begin{equation*}
\begin{aligned}
\frac{1}{n}\sum_{j\in\mathbb Z}
\varphi^1\Big(\frac{j}{n}\Big)
\varphi^2\Big(\frac{j}{n}+a_n\Big)
&= \frac{1}{n}\sum_{j\in\mathbb Z}
\int_{\mathbb R^2}
\widehat \varphi^1(\xi_1) \widehat \varphi^2(\xi_2)
e^{2\pi\mathsf i \frac{j}{n} \xi_1} 
e^{2\pi\mathsf i (\frac{j}{n}+a_n)\xi_2} d\xi_1 d\xi_2 \\
&= \frac{1}{n}\int_{\mathbb R^2}
\widehat \varphi^1(\xi_1) \widehat \varphi^2(\xi_2)
\delta\Big(\frac{\xi_1+\xi_2}{n}\Big)
e^{2\pi\mathsf i a_n\xi_2} d\xi_1 d\xi_2 \\
&= \frac{1}{2n}\int_{\mathbb R^2}
\widehat \varphi^1\Big(\frac{\eta_1+\eta_2}{2}\Big) 
\widehat \varphi^2\Big(\frac{\eta_1-\eta_2}{2}\Big)
\delta\Big(\frac{\eta_1}{n}\Big)
e^{2\pi\mathsf i a_n\frac{\eta_1-\eta_2}{2}} d\eta_1 d\eta_2 \\
&= \frac{1}{2} \int_\mathbb R
\widehat \varphi^1(\eta_2/2) 
\widehat \varphi^2(-\eta_2/2) 
e^{2\pi\mathsf i (-a_n/2)\eta_2} d\eta_2 .
\end{aligned}
\end{equation*}
By a standard Fourier inversion formula, we notice that the utmost right-hand
side of the last display is rewritten as $\psi(-a_n/2)$ for some $\psi\in\mathcal{S}(\mathbb R)$
which satisfies 
\begin{equation*}
\widehat \psi (\xi)
=\frac{1}{2} \widehat\varphi^1(\xi/2)
\widehat\varphi^2(-\xi/2). 
\end{equation*}
Since $|\psi(x)|\to0$ as $|x|\to \infty$, we complete the proof.
\end{proof}

\section{Equipartition of Energy} 
\label{sec:energy-fluctuations}

Here let us give a proof of \cref{lem-equipart}. 
In the following we \Add{take a time-varying test function $\varphi(j,t)$ satisfying $\sum_{j}\varphi(j,t)=0$ for any $t\ge 0$.}  
The following results are independent of $\sigma = \pm 1$, so that we omit the mode dependence.

\begin{lemma}
  \label{lem:equip-energy-0}
We have that 
\begin{equation}
  \label{eq:equip-0}
\begin{aligned}
& \mathbb E_n \bigg[ \sup_{0\le t\le T} \bigg|
\int_0^t \sum_{j\in\mathbb Z} \big(r_{j+1}(V'_n(r_{j+1}) - V'_n(r_{j})) +
p_{j}(p_{j+1}-p_{j}) \big)(s) \varphi(j, s)  ds \bigg|^2
\bigg] \\
&\quad \lesssim \frac{1}{n}
\left( \|\varphi \|^2_{L^2(\mathbb R)} + \| \partial_x \varphi \|^2_{L^2(\mathbb R)}\right).
\end{aligned}
\end{equation}
\end{lemma}

\begin{proof}
  Note that
\begin{equation}
\label{eq:generator_action_product_rp}
\begin{aligned}
L(r_{j+1} p_{j})
&= 
\frac{\gamma}{2}(\Delta p_{j})r_{j+1} 
+ \alpha
\big[r_{j+1}(V_n'(r_{j+1})-V_n'(r_j))  +p_{j}(p_{j+1}-p_{j}) \big].
\end{aligned}
\end{equation}
Set 
\begin{equation*}
Y_t(\varphi)
= \frac{1}{ n^2} \sum_{j\in\mathbb Z}r_{j+1}(t)p_j(t) \varphi(j, t) .  
\end{equation*}
Observe that
\begin{equation}
\label{eq:SY}
n^2S Y_t(\varphi) = \frac{\gamma}{2} \sum_{j\in\mathbb Z}
(\Delta p_{j}(t)) r_{j+1}(s) \varphi(j, t). 
\end{equation}
Then, by Dynkin's martingale formula applied to this fluctuation field, we have that 
\begin{equation}
\label{eq:equipartition_energy_martingale}
\begin{aligned}
\mathfrak{M}^Y_t(\varphi)
&=Y_t (\varphi)
- Y_0(\varphi) 
-\int_0^t(\partial_s+ n^2 L)Y_s(\varphi)ds\\
&= Y_t(\varphi) -Y_0(\varphi) 
- \int_0^t  n^2S Y_s(\varphi) ds \\
&\quad- \int_0^t \sum_{j\in\mathbb Z} \alpha \big(r_{j+1}(V_n'(r_{j+1})-V_n'(r_j))
+p_{j}(p_{j+1}-p_{j}) \big)(s)
\varphi(j, s)  ds \\
&\quad -\frac{ v^\sigma_n}{ n^2} \int_0^t\sum_{j\in\mathbb Z} \big(r_{j+1}p_j\big)(s) (\partial_x \varphi) (j, s)   ds 
\end{aligned}
\end{equation}
is a mean-zero martingale with respect to the natural filtration.

Here, note that the quadratic variation of $\mathfrak M^Y_t(\varphi)$ is computed as 
\begin{equation*}
\begin{aligned}
\langle \mathfrak M^Y (\varphi) \rangle_t 
&= \frac{\gamma}{2 n^2} \int_0^t \sum_{j\in\mathbb Z} r_{j+1}^2(p_j-p_{j+1})^2
\big( \nabla^n \varphi (j,s)\big)^2 ds. 
\end{aligned}
\end{equation*}
Hence by Doob's inequality, we have that 
\begin{equation*}
\mathbb E_n\Big[\sup_{0\le t\le T} \mathfrak M^Y_t(\varphi)^2 \Big] 
\le 4\mathbb E_n\big[\mathfrak M^Y_T(\varphi)^2 \big]
\lesssim \frac{T}{ n} \| \partial_x \varphi \|^2_{L^2(\mathbb R)} .
\end{equation*}
On the other hand, as a consequence of the bound \eqref{eq:h-1}, 
\begin{equation*}
\begin{aligned}
\mathbb E_n \bigg[\sup_{0\le t\le T} \bigg| 
\int_0^t n^2 S Y_s(\varphi) ds \bigg|^2\; \bigg] 
\le 24 n^2 \int_0^T \mathbb E\left[ Y_s(\varphi) S Y_s(\varphi)\right]  ds
\lesssim \frac{T}{n}  \|\varphi\|^2_{L^2(\mathbb R)} 
\end{aligned}
\end{equation*}
and by the Cauchy-Schwarz inequality
\begin{equation*}
\begin{aligned}
\mathbb E_n \bigg[\sup_{0\le t\le T} \bigg| 
\frac{v^\sigma_n}{n^2} \int_0^t\sum_{j\in\mathbb Z} \big(r_{j+1}p_j\big)(s)
(\partial_x \varphi)(j,s) ds \bigg|^2 \bigg]
\lesssim \frac{T^2  c_2^2\alpha^2}{n} \|\partial_x \varphi \|^2_{L^2(\mathbb R)} .
\end{aligned}
\end{equation*}
Moreover, by Lemma \ref{lem:gen_bound},  
\begin{equation*}
\begin{aligned}
\mathbb E_n\Big[\sup_{0\le t\le T} Y_t(\varphi)^2 \Big] 
\lesssim \frac{1}{n^4}
T^2 n^3 \| \partial_x \varphi \|^2_{L^2(\mathbb R)} = \frac{T^2}{n} \| \partial_x \varphi \|^2_{L^2(\mathbb R)}. 
\end{aligned}
\end{equation*}
This concludes the proof of \eqref{eq:equip-0}. 
\end{proof}

\begin{lemma}
\label{lem:equipartition_energy}
We have that 
\begin{equation*}
\begin{aligned}
& \mathbb E_n \bigg[ \sup_{0\le t\le T} \bigg|
\int_0^t \sum_{j\in\mathbb Z} \big(c_2 \overline{r_{j+1}(r_{j+1}-r_{j})}
+ \overline{p_{j}(p_{j+1}-p_{j})}
\Add{+ \ve_n\frac {c_3}{c_2\beta} \overline{r_j}}
\big)(s) \varphi(j, s)  ds \bigg|^2
\bigg] \\
&\quad \lesssim \frac{1}{n}
\left( \|\varphi \|^2_{L^2(\mathbb R)} + \| \partial_x \varphi \|^2_{L^2(\mathbb R)}\right).
\end{aligned}
\end{equation*}
\end{lemma}

\begin{proof}

  By \eqref{eq:generator_action_product_rp}, since the corresponding average in null,
\begin{equation}
 \label{eq:7}
\begin{split}
&r_{j+1}(V'_n(r_{j+1}) - V'_n(r_{j})) + p_{j}(p_{j+1}-p_{j}) \\
&=\overline{r_{j+1}(V'_n(r_{j+1}) - V'_n(r_{j}))} 
  +\overline{p_{j}(p_{j+1}-p_{j})}\\
& = c_2 \overline{r_{j+1}(r_{j+1}- r_j)} +
\frac{c_3}{2} \ve_n \overline{r_{j+1}(r_{j+1}^2- r_j^2)}+\overline{p_{j}(p_{j+1}-p_{j})}
+  O(\ve_n^2).
\end{split}
\end{equation}
We also have the identity
\begin{equation}
\label{eq:r2p}
\begin{aligned}
A(r_{j}^2p_j)
&=\overline{r_{j}^2 \big(V'_n(r_{j+1})-V'_n(r_j)\big)}
+ \overline{2r_{j}(p_{j}-p_{j-1})p_j} \\
&=  c_2 \overline{r_{j}^2 (r_{j+1}-r_j)} + \frac{c_3}{2} \ve_n \overline{r_{j}^2 (r_{j+1}^2-r_j^2)}
+ \overline{2r_{j}(p_{j}-p_{j-1})p_j} 
+ O(\varepsilon_n^2)\\
&= \kh{-}c_2 \left(\overline{r_j^3} - \overline{r_{j+1}^3}\right) - c_2 \overline{r_{j+1}\left(r_{j+1}^2 - r_{j}^2\right)}
+ \frac{c_3}{2} \ve_n \overline{r_{j}^2 (r_{j+1}^2-r_j^2)}
+ \overline{2r_{j}(p_{j}-p_{j-1})p_j}
+ O(\varepsilon_n^2).
\end{aligned}
\end{equation}
Observe that
\begin{equation}
  \label{eq:17}
  \begin{split}
\overline{2r_{j}(p_{j}-p_{j-1})p_j}
&= 2 \overline{r_{j}p_{j}^2} - 2\overline{r_jp_{j-1}p_j}\\
&= 2 \overline{r_{j}}\overline{p_{j}^2} 
\Add{+ 2 \overline{r_{j}} E_{\nu_n}(p_j^2)} + 2 E_{\nu_n}(r_j) \overline{p_{j}^2}
- 2 \overline{r_{j}} p_{j}p_{j-1} .
  \end{split}
\end{equation}
Then we have
\begin{equation}
\label{eq:18}
\begin{split}
\frac{c_3}{2} \ve_n \overline{r_{j+1}\left(r_{j+1}^2 - r_{j}^2\right)} 
&=\Add{\ve_n\frac {c_3}{c_2\beta} \overline{r_j}}
+ \ve_n\frac {c_3}{c_2}\overline{r_{j}}\overline{p_{j}^2}
-  \ve_n \frac {c_3}{c_2}\overline{r_{j}} p_{j}p_{j-1} 
-\ve_n^2 \frac{c_3^2}{c_2^3\beta} \overline{(p_{j}^2 - p_{j}p_{j-1})} \\
&\quad+  \frac{c_3}{2}\varepsilon_n \left(\overline{r_j^3} - \overline{r_{j+1}^3}\right) 
-\frac{c_3}{2c_2}\varepsilon_n A(r_j^2p_j) 
+ O(\varepsilon_n^2).
\end{split}
\end{equation}
The terms in the right-hand side of \eqref{eq:18} with a factor $\ve_n^2$ can be handled by a straightforward Schwarz inequality.
On the other hand, the fifth (resp. sixth) term turns out to be negligible, by summation by parts (resp. by \cref{prop:kipnis_varadhan_estiamte}). 
Hence, we are left with the first three terms of the right-hand side of \eqref{eq:18}.
The second and the third terms are also negligible from \cref{lem:cub}, and thus the conclusion follows by virtue of \eqref{eq:7} and \cref{lem:equip-energy-0}.

\end{proof}


\begin{lemma}\label{lem:cub}
There exists some $C_T>0$ such that 
\begin{equation}
\label{eq:15b}
\begin{split}
\mathbb E_n\bigg[ \sup_{0\le t\le T} \bigg| \varepsilon_n 
\int_0^t \sum_{j\in\mathbb Z} \overline{r_{j}}\; \overline{p_{j}^2}(s) \varphi(j,s) ds 
\bigg|^2 \bigg]
\le \frac{C_T}{n} \left(\| \varphi\|^2_{L^2(\mathbb R)}  +
  \|\partial_x \varphi\|^2_{L^2(\mathbb R)} \right), \\
 \mathbb E_n\bigg[ \sup_{0\le t\le T} \bigg|
\varepsilon_n 
\int_0^t \sum_{j\in\mathbb Z} \overline{r_{j}}\; p_{j}p_{j+1}(s) \varphi(j,s) ds 
\bigg|^2 \bigg]
\le \frac{C_T}{n} \left(\| \varphi\|^2_{L^2(\mathbb R)}  +
  \|\partial_x \varphi\|^2_{L^2(\mathbb R)} \right). 
  \end{split}
\end{equation}
\end{lemma}
\begin{proof}
Both estimates are obtained as a direct consequence of \cref{lem:consequence_one_block_estimate} applied to the function $G(p_j) = \overline{p_{j}^2}$ and $G(p_j) = p_{j}p_{j+1}$.  
\end{proof}

\Add{
\begin{lemma}
\label{lem:cubic_term_characterization}
\begin{equation*}
\lim_{n\to\infty} 
\mathbb E_n\bigg[\sup_{0\le t\le T}
\bigg|
\varepsilon_n \int_0^t \sum_{j\in\mathbb Z} \Big(\overline{r_j^3}- {\frac{3}{ \beta c_2}}\overline{r_j}\Big)(s)
\varphi(j,s) ds \bigg|^2 \bigg]
=0.
\end{equation*}
\end{lemma}
\begin{proof}
In view of \eqref{eq:18}, we have already shown that
\begin{equation*}
\begin{aligned}
\lim_{n\to\infty}
\mathbb E_n \bigg[ \sup_{0\le t\le T} \bigg|
\varepsilon_n\int_0^t \sum_{j\in\mathbb Z} \Big(
\frac{c_3}{2}\overline{r_{j+1}(r_{j+1}^2-r_j^2)}
 - \frac{c_3}{c_2\beta}\overline{r_j} \Big)(s) \varphi(j, s)  ds \bigg|^2 \bigg] =0.
\end{aligned}
\end{equation*}
by the assumption $\sum_{j}\varphi(j,t)=0$, 
we can always add a constant of order $\ve_n$
and obtain 
\begin{equation*}
\begin{aligned}
\lim_{n\to\infty}
\mathbb E_n \bigg[ \sup_{0\le t\le T} \bigg|
\varepsilon_n\int_0^t \sum_{j\in\mathbb Z} \Big(
\frac{c_3}{2}r_{j+1}(r_{j+1}^2-r_j^2) 
 - \frac{c_3}{c_2\beta}r_j \Big)(s) \varphi(j,s) ds \bigg|^2
\bigg] =0.
\end{aligned}
\end{equation*}
Here, note that 
\begin{equation*}
r_{j+1}r_j^2
= \overline r_{j+1} 
\overline{r_j^2} 
+ E_{\nu_n}(r_j^2) \overline r_{j+1} 
+ O(\varepsilon_n) 
\end{equation*}
where we used the fact that $E_{\nu_n}[r_j]=O(\varepsilon_n)$. 
Since $E_{\nu_n}(r_j^2)=(c_2\beta)^{-1}+ O(\varepsilon_n^2)$, summation-by-parts and \cref{lem:rr2} conclude the proof. 
\end{proof}
}

Now, we are in a position to give a proof of equipartition of the energy (\cref{lem-equipart}).

\begin{proof}[Proof of \cref{lem-equipart}] 
In view of \cref{lem:equipartition_energy}, it suffices to show the following: 
\begin{equation}\label{eq:equip}
\limsup_{n\to\infty}
\mathbb E_n \bigg[ \sup_{0\le t\le T} \bigg|
\int_0^t \sum_{j\in\mathbb Z} \big(p_jp_{j+1}
- c_2 r_jr_{j+1}\big)(s) \varphi(j,s) ds \bigg|^2 \bigg] 
= 0. 
\end{equation}
However, we notice that 
\begin{equation*}
\begin{aligned}
2(p_jp_{j+1}
- c_2 r_jr_{j+1})
= \xi^+_j\xi^-_{j+1} + \xi^-_j\xi^+_{j+1}. 
\end{aligned}
\end{equation*}
Then \eqref{eq:equip} is a direct consequence of \cref{lem:quadratic_field_with_wrong_velocity_cross_fields}. 
\end{proof}

\section{Tightness}
\label{sec:tightness}
In this section, we show that each sequence in the decomposition \eqref{eq:martingale_decomposition} is tight. 

\subsection{Preliminaries}
In this part, for readers' convenience, we recall some basic notions of the Skorohod space and the tightness of a sequence in the c\`{a}dl\`{a}g space. 
To begin with a general setting, let $E$ be a complete, separable metric space, endowed with a distance $d_E$.
Let $D([0,T],E)$ be the space of all right continuous functions with left limits taking values on $E$. 
Let $\lambda $ be the set of all strictly increasing continuous functions $\lambda$ from $[0,T]$ into itself. 
Then, we define 
\begin{equation*}
\| \lambda \| 
= \sup_{s\neq t} \bigg| \log \frac{\lambda(t)-\lambda(s)}{t-s} \bigg|
\end{equation*}
and define for each $X,Y\in D([0,T],E)$ 
\begin{equation*}
d(X,Y)
= \inf_{\lambda \in\Lambda} \max \big\{ \| \lambda \|, \sup_{0\le t\le T} d_E(X_t, Y_{\lambda(t)}) \big\}. 
\end{equation*}
Then it is known that the Skorohod space $D([0,T],E)$ endowed with the metric $d$ is a complete separable metric space, see \cite[Chapter 3]{billingsley1968convergence}. 
Next, in order to characterize the convergence of a sequence of paths in the Skorohod space, we make use of the following modified modulus of continuity: for each $X=\{X_t:t\in[0,T]\} \in D([0,T],E)$, set 
\begin{equation*}
w^\prime_X(\gamma)
= \inf_{\{t_i\}_{0\le t \le N}}
\max_{0\le i <N} \sup_{t_i\le s<t <t_{i+1}}
d_E(X_s, X_t),
\end{equation*}
where the first infimum is taken over all partitions $\{ t_i\}_{0\le i\le N}$ of the interval $[0,T]$ such that $0=t_0<t_1<\cdots <t_N=T$ and $t_i-t_{i-1}>\gamma$ for each $i=1,\ldots, N$. 
Then, the relative compactness of a sequence in the Skorohod space is characterized by the following Prohorov's theorem \cite[Theorem 4.1.3]{kipnis1998scaling}. 

\begin{proposition}
\label{prop:prohorov}
Let $\{ \mathbb{P}_n\}_n$ be a sequence of probability measures on $D([0,T],E)$. 
The sequence is relatively compact if, and only if,
\begin{itemize}
\item[(1)] For each $t\in[0,T]$ and each $\delta>0$ there exists a compact set $K(t,\delta)$ in $E$ such that $\mathbb{P}_n(X_t \notin K(t,\delta))\le \delta$. 
\item[(2)] For each $\delta>0$, we have $
\lim_{\gamma\to0} \limsup_{n\to \infty}
\mathbb{P}_n (w^\prime_X(\gamma) > \delta) = 0$. 
\end{itemize}
\end{proposition}

Here note that the modulus of continuity has the bound $w^\prime_X(\gamma) \le w_X(2\gamma)$ where 
\begin{equation*}
w_X(\gamma) = \sup_{|t-s|\le \gamma} d_E(X_s,X_t)
\end{equation*}
for each $X\in D([0,T],E)$. 
Therefore, to show that a sequence in the Skorohod space is relatively compact, which is equivalent to the sequence being tight since the space is complete and separable, it suffices to show the following condition (2') instead of the condition (2) in Proposition \ref{prop:prohorov}. 

\begin{itemize}
\item[(2')]  For each $\delta>0$, we have $
\lim_{\gamma\to0} \limsup_{n\to \infty}
\mathbb{P}_n (w_X(\gamma) > \delta) = 0$. 
\end{itemize}

{Note that once the condition (2') is verified, combined with the condition (1) of Proposition \ref{prop:prohorov}, then all limit points of a sequence $\{\mathbb{P}_n\}_n$ are concentrated on continuous paths.
}  

Now we return to our current situation and recall the martingale decomposition \eqref{eq:martingale_decomposition}. 
Our central aim is to show the tightness of each sequence. 

\begin{lemma}
\label{lem:tightness}
For each $\sigma=\pm$, the sequences $\{\widetilde{\mathcal{X}}^{\sigma,n}_t : t \in [0, T ] \}_{ n \in \mathbb{N} } $, $\{ \mathcal{M}^{\sigma,n}_t : t \in [0, T ] \}_{ n \in \mathbb{N} } $, $\{ \mathcal{S}^{\sigma,n}_t : t \in [0, T ] \}_{ n \in \mathbb{N} } $
and $\{ \mathcal{B}^{\sigma,n}_t : t \in [0, T ] \}_{ n \in \mathbb{N} } $,
when the processes start from the invariant measure $\nu_n$,
are tight with respect to the Skorohod topology on $D([0,T],\mathcal{S}^\prime (\mathbb{R})) $.
\end{lemma}

Here, note that the space of Schwartz distributions $\mathcal{S}^\prime(\mathbb{R})$ is metrizable, which turns out to be separable and complete with respect to the strong topology.
(See \cite[Section 2.3]{gonccalves2014nonlinear} for a precise description of the topology.) 
To prove the tightness of a sequence of processes, the following Mitoma's criterion \cite[Theorem 4.1]{mitoma1983tightness} is helpful.

\begin{proposition}[Mitoma's criterion]
\label{Mitoma}
A sequence of $\mathcal{S}^\prime (\mathbb{R} ) $-valued processes $\{ \mathcal{Y}^n_t : t \in [0, T ] \}_{ n \in \mathbb{N} } $ with trajectories in $D ([0, T ] , \mathcal{S}^\prime (\mathbb{R} ) ) $ is tight with respect to the uniform topology if, and only if, the sequence $\{ \mathcal{Y}^n_t (\varphi ) : t \in [0, T ] \}_{ n \in \mathbb{N}}$ of real-valued processes is tight with respect to the Skorohod topology on $D([0,T],\mathbb{R}) $ for any $\varphi \in \mathcal{S}(\mathbb{R}) $.   
\end{proposition}

In addition, we will use the following continuity criterion.
(See for example \cite[Theorem 1.2.1]{revuz2013continuous}.)

\begin{proposition}[The Kolmogorov-Chentsov criterion]
\label{prop:kolmogorov_chentsov}
Let $\{ X_t : t\in [0,T]\}$ be a Banach-valued process for which there exists constants $\kappa,\gamma_1 \gamma_2>0$ such that 
\begin{equation*}
\mathbb{E} \big[ \big\|X_t-X_s\big\|^{\gamma_1} \big]
\le \kappa | t-s|^{1+\gamma_2},
\end{equation*}
for any $s,t\in[0,T]$. 
Here $\| \cdot\|$ denotes the norm of the Banach space on which the process takes values. 
Then, there is a modification $\tilde{X}$ of $X$ such that 
\begin{equation*}
\mathbb{E} \bigg[ \bigg( \sup_{s\neq t} \frac{\|\tilde{X}_t -\tilde{X}_s\|}{|t-s|^{\alpha}} \bigg)^{\gamma_1} \bigg] < +\infty
\end{equation*}
for any $\alpha \in [0, \gamma_2/\gamma_1)$. 
In particular, the paths of $\tilde{X}$ are almost-surely $\alpha$-H\"{o}lder continuous. 
\end{proposition}

In what follows, we prove Lemma \ref{lem:tightness}. 
With the help of Mitoma's criterion, it suffices to show the tightness of sequences
$\{ \widetilde{\mathcal{X}}^{\sigma,n}_t (\varphi):t\in [0,T]\}_{n\in\mathbb N}$, $\{ \mathcal{S}^{\sigma,n}_t (\varphi) : t \in [0, T ] \}_{ n \in \mathbb{N} } $, $\{ \mathcal{B}^{\sigma,n}_t (\varphi) : t \in [0, T ] \}_{ n \in \mathbb{N} } $ and $\{ \mathcal{M}^{\sigma,n}_t (\varphi) : t \in [0, T ] \}_{ n \in \mathbb{N}}$ in $D ([0,T],\mathbb{R})$ for any given test function $\varphi \in \mathcal{S} (\mathbb{R})$. 
In order to prove the tightness of a real-valued sequence $\{X^{\sigma,n}_t:t\ge0 \}_n$ in $D([0,T],\mathbb{R})$, according to Prohorov's theorem, recall that it suffices to show two conditions, namely the condition (1) in Proposition \ref{prop:prohorov} and the condition (2'): 
\begin{equation}
\label{eq:tighness_condition_prohorov}
\lim_{\delta\to0} \limsup_{n\to \infty}
\mathbb{P}_n \Bigg( \sup_{\substack{|t-s|\le \delta\\ 0\le s,t \le T}} | X^n_t-X^n_s| > \varepsilon \Bigg)
= 0,
\end{equation}
for any $\varepsilon>0$. 
The condition (1) of Proposition \ref{prop:prohorov} on fixed times easily follows for our sequences. 
Hence, in what follows, our task is to verify the condition \eqref{eq:tighness_condition_prohorov} for the sequences $\{ \mathcal{X}^{\sigma,n}_t (\varphi) : t \in [0, T ] \}_{ n \in \mathbb{N} } $, $\{ \mathcal{S}^{\sigma,n}_t (\varphi) : t \in [0, T ] \}_{ n \in \mathbb{N} } $, $\{ \mathcal{B}^{\sigma,n}_t (\varphi) : t \in [0, T ] \}_{ n \in \mathbb{N} } $ and $\{ \mathcal{M}^{\sigma,n}_t (\varphi) : t \in [0, T ] \}_{ n \in \mathbb{N}}$. 
For the initial field $\{ \mathcal{X}^{\sigma,n}_0\}_n$, it is enough to observe that by characteristic functions we can show that it converges to a Gaussian field, and in particular it is tight. 
Thus, in what follows, we focus on the tightness of the martingale, symmetric and antisymmetric parts, from which the tightness of the fields $\{ \mathcal{X}^{\sigma,n}_\cdot\}_n$ is deduced.

\subsection{Martingale Part}
To show the tightness of the martingale part $\mathcal{M}^{\sigma,n}_\cdot$, we show the following estimate on the fourth moment. 

\begin{lemma}
For each smooth compactly supported function $\varphi$, and for each $s,t\in[0,T]$ such that $s<t$, we have that  
\begin{equation*}
\mathbb E_n \Big[(\mathcal{M}^{\sigma,n}_t(\varphi)-\mathcal{M}^{\sigma,n}_s(\varphi))^4 \Big] 
\lesssim  |t-s|^2 + n^{-3}|t-s|  .  
\end{equation*}
\end{lemma}
The proof is analogous to \cite[Lemma 5.5]{gonccalves2023derivation}, so that we omit the proof here. 
With this fourth moment estimate at hand, we can show the tightness of the martingale part. 
Indeed, note that 
\begin{equation*}
\begin{aligned}
\mathbb P_n 
\bigg( \sup_{\substack{|s-t|\le \delta \\0\le s,t\le T}} 
| \mathcal{M}^{\sigma,n}_t(\varphi)-\mathcal{M}^{\sigma,n}_s(\varphi)| >\varepsilon \bigg)
&\le \varepsilon^{-4} \mathbb E_n \bigg[ \sup_{\substack{|s-t|\le \delta \\0\le s,t\le T}} 
| \mathcal{M}^{\sigma,n}_t(\varphi)-\mathcal{M}^{\sigma,n}_s(\varphi)| \bigg] \\
&\lesssim \varepsilon^{-4} \delta^{-1} \mathbb E_n\big[(\mathcal{M}^{\sigma,n}_\delta(\varphi)^4)]
\end{aligned}
\end{equation*}
where we used Doob's inequality and stationarity of the process. 
Since the utmost right-hand side of the last display vanishes as $n\to\infty$ and then $\delta\to0$, the condition \eqref{eq:tighness_condition_prohorov} is deduced and thus the sequence $\{ \mathcal{M}^{\sigma,n}_\cdot(\varphi) \}_n$ is tight for each $\sigma=\pm$, provided the test function $\varphi$ is smooth, compactly supported. 
To extend the class of test functions, it suffices to approximate an arbitrary function in $\mathcal{S}(\mathbb R)$ in the martingale decomposition \eqref{eq:martingale_decomposition}.

\subsection{Symmetric Part}
We show the tightness of the symmetric part. 
Let 
\begin{equation*}
\mathcal{S}^{\sigma,n}_t(\varphi)
= \frac{\gamma}{4} \int_0^t 
\mathcal X^{\sigma,n}_s (\partial_x^2\varphi_{v^\sigma_ns} ) ds
\end{equation*}
for each $\sigma\in\{+,-\}$. 
Note that by a direct $L^2$-estimate, we have that 
\begin{equation*}
\begin{aligned}
\mathbb E_n \Big[|\mathcal{S}^{\sigma,n}_{t_1}(\varphi) - \mathcal{S}^{\sigma,n}_{t_2}(\varphi)|^2 \Big] 
&\lesssim |t_1-t_2|\int_{t_1}^{t_2} \frac{1}{n} \sum_{j\in\mathbb Z}
E_{\nu_n}[(\overline \xi^\sigma_j)^2] 
\big( \partial_x^2 \varphi(\tfrac{j}{n}+v^\sigma_ns) \big)^2 ds \\
&\lesssim |t_1-t_2|^2 \| \partial_x^2 \varphi\|^2_{L^2(\mathbb R)}. 
\end{aligned}
\end{equation*}
By continuity of $\mathcal{S}^{\sigma,n}_\cdot$ and the Kolmogorov-Chentsov criterion, the condition \eqref{eq:tighness_condition_prohorov} is deduced. 
Hence the tightness of the symmetric part follows immediately.

\subsection{Antisymmetric Part}
Now, we show the tightness of the antisymmetric part.
Define 
\begin{equation*}
\mathcal{B}^{\sigma,n}_t(\varphi)
= -\sigma\alpha \frac{c_3}{8c_2^2} 
\int_0^t\sum_{j\in\mathbb Z}
\overline{\xi}^{\sigma}_j(s) \overline{\xi}^\sigma_{j+1}(s)
\partial_x \varphi(\tfrac{j}{n}+ v^\sigma_ns) ds
\end{equation*}
for each $\sigma\in\{+,-\}$. 
Tightness of the antisymmetric part follows from the second-order Boltzmann-Gibbs principle.  
Indeed, by Proposition \ref{prop:2bg}, 
\begin{equation*}
\begin{aligned}
&\mathbb E_n \Bigg[ \bigg| \mathcal{B}^{\sigma,n}_{t_2}(\varphi) 
-\mathcal{B}^{\sigma,n}_{t_1}(\varphi)
+ \sigma\alpha\frac{c_3}{8c_2^2} \int_{t_1}^{t_2} \sum_{j\in\mathbb Z}
\big( (\overrightarrow\xi^\sigma)^\ell_j \big)^2 
\partial_x \varphi(\tfrac{j}{n}+ v^\sigma_ns)  ds
\bigg|^2 \Bigg] \\
&\quad \lesssim \bigg( \frac{(t_2-t_1)\ell}{n} + 
{\frac{t_2-t_1}{\ell}} \bigg) \| \partial_x \varphi\|^2_{L^2(\mathbb R)} .
\end{aligned}
\end{equation*}
On the other hand, by a crude $L^2$-estimate, we have that 
\begin{equation*}
\mathbb E_n \Bigg[ \bigg| \int_{t_1}^{t_2} \sum_{j\in\mathbb Z} 
\big( (\overrightarrow \xi^\sigma)^\ell_j \big)^2 
\partial_x \varphi(\tfrac{j}{n}+ v^\sigma_ns)  ds \bigg|^2 \Bigg]
\lesssim \frac{(t_2-t_1)^2n}{\ell} \| \partial_x \varphi\|^2_{L^2(\mathbb R)} 
\end{equation*}
since $E_{\nu_n}[((\overrightarrow\xi^\sigma)^\ell_j)^4]\lesssim \ell^{-2}$. 
When $1/n^2\le t_2-t_1\le 1$, we take $\ell$ with order $(t_2-t_1)^{1/2}n$ to obtain 
\begin{equation*}
\mathbb E_n \big[ | \mathcal{B}^{\sigma,n}_{t_2}(\varphi)
- \mathcal{B}^{\sigma,n}_{t_1}(\varphi)|^2 \big]
\lesssim (t_2-t_1)^{3/2} \| \partial_x \varphi\|^2_{L^2(\mathbb R)}. 
\end{equation*}
On the other hand, when $t_2-t_1 \le 1/n^2$, by a direct estimate we have that 
\begin{equation*}
\mathbb E_n \big[ | \mathcal{B}^{\sigma,n}_{t_2}(\varphi)
- \mathcal{B}^{\sigma,n}_{t_1}(\varphi)|^2 \big]
\lesssim (t_2-t_1)^{2}n \| \partial_x \varphi\|^2_{L^2(\mathbb R)}
\lesssim (t_2-t_1)^{3/2} \| \partial_x \varphi\|^2_{L^2(\mathbb R)}. 
\end{equation*}
Hence, the antisymmetric part is tight by the Kolmogorov-Chentsov criterion and continuity of the process $\mathcal{B}^{\sigma,n}_\cdot$.

\section{Identification of the limit points}
\label{sec:identification_limit_points}
Recall the martingale decomposition \eqref{eq:martingale_decomposition}. 
Then, as a consequence of Lemma \ref{lem:tightness}, the sequences
$\{ \widetilde{\mathcal{X}}^{\sigma,n}_\cdot\}_n$, $\{\mathcal{M}^{\sigma,n}_\cdot\}_n$,
$\{\mathcal{S}^{\sigma,n}_\cdot\}_n$ and $\{\mathcal{B}^{\sigma,n}_\cdot\}_n$ are tight with respect to the
Skorohod topology in $D([0,T],\mathcal{S}'(\mathbb R))$ for each $\sigma=\pm$. 
Let $\mathscr Q^{\sigma,n}$ be the distribution of 
\begin{equation*}
\{ (\widetilde{\mathcal{X}}^{\sigma,n}_t, 
\mathcal{M}^{\sigma,n}_t, 
\mathcal{S}^{\sigma,n}_t, 
\mathcal{B}^{\sigma,n}_t) : t\in [0,T] \} . 
\end{equation*}
Then, there exists a subsequence $n$, which is denoted by the same letter with abuse of notation, such that $\{ \mathscr Q^{\sigma,n}\}_n$ converges to a limit point $\mathscr Q^\sigma$. 
Let $\mathcal{X}^\sigma$, $\mathcal{M}^\sigma$, $\mathcal{S}^\sigma$ and $\mathcal{B}^{\sigma}$ be the respective limits in distribution of each component. 
Since the tightness is shown in the uniform norm, these limiting processes have continuous trajectories, almost surely.
Hereafter we characterize these limit points as stationary energy solution of the stochastic Burgers equation \eqref{eq:drift-sbe}.  
Since the solution of the stochastic Burgers equation is unique-in-law, then convergence along the full sequence follows.

\subsection{Martingale Part}
Recall that the quadratic variation of the martingale $\mathcal{M}^{\sigma,n}$ is given by \eqref{eq:quadratic_variation}.  
Then, by Markov's inequality and Schwarz's inequality, for each $\sigma=\pm$ we have that 
\begin{equation*}
\lim_{n\to\infty}
\mathbb P_n
\bigg( \sup_{0\le t\le T} \Big| 
\langle \mathcal{M}^{\sigma,n}(\varphi)\rangle_t 
- \beta^{-1} \gamma t \| \partial_x \varphi \|^2_{L^2(\mathbb R)} 
\Big| > \varepsilon \bigg) 
= 0
\end{equation*}
for any $\varepsilon>0$.  
In particular, for each $\sigma=\pm$, the process $\{ \langle \mathcal{M}^{\sigma,n}(\varphi)\rangle_t : t\in [0,T]\}$ converges in distribution on $D([0,T],\mathbb R)$ to a deterministic path $\{ \beta^{-1}\gamma t\| \partial_x \varphi\|^2_{L^2(\mathbb R)}: t\in [0,T]\}$ as $n\to \infty$. 
Then, we can show that any limit point $\mathcal{M}^\sigma$ is a continuous martingale with quadratic variation $\beta^{-1}\gamma t\|\partial_x \varphi\|^2_{L^2(\mathbb R)}$. 
Indeed, note that the limit point $\mathcal{M}^\sigma$ is a martingale since it is obtained as a limit of martingales with respect to the uniform topology, and that we have that bound 
\begin{equation*}
\mathbb E_n \Big[ 
\sup_{0\le s \le t} \big| \mathcal{M}^{\sigma,n}_s(\varphi)
- \mathcal{M}^{\sigma,n}_{s-}(\varphi)\big| \Big]
\le 2 \mathbb E_n \Big[ 
\sup_{0\le s \le t} \big| \mathcal{M}^{\sigma,n}_s(\varphi)
\big|^2 \Big]^{1/2} 
\lesssim 2 \mathbb E_n \big[\langle \mathcal{M}^{\sigma,n}(\varphi)\rangle_t  \big]^{1/2}
\end{equation*}
by the triangle inequality and Doob's inequality. 
Here, we notice that the utmost right-hand side of the last display is bounded by a constant which is independent of $n$.
Therefore, by \cite[Corollary VI.6.30]{jacod2013limit}, the convergence $(\mathcal{M}^{\sigma,n}(\varphi), \langle\mathcal{M}^{\sigma,n}(\varphi)\rangle) \to (\mathcal{M}^{\sigma}(\varphi), \langle\mathcal{M}^{\sigma}(\varphi)\rangle)$ in distribution follows. 
Combining with the convergence of the quadratic variation, we conclude that $\langle \mathcal{M}^\sigma (\varphi)\rangle_t= \beta^{-1} \gamma t\| \partial_x \varphi\|^2_{L^2(\mathbb R)}$ for each $t \in [0,T]$ and $\sigma\in \{+,-\}$.

\subsection{Symmetric Part}
For the symmetric part, the following identity is straightforward:     
\begin{equation*}
\mathcal{S}^\sigma_t(\varphi)
= \frac{\gamma}{4} \int_0^t \mathcal{X}^\sigma_s(\partial_x^2 \varphi) ds. 
\end{equation*}

\subsection{Antisymmetric Part}
Finally, we deal with the antisymmetric part. 
Since the sequence $\{ \widetilde{\mathcal{X}}^{\sigma,n}\}_n$ is tight,
we have that $\lim_{n\to\infty}\widetilde{\mathcal{X}}^{\sigma,n}_\cdot
=\mathcal{X}^\sigma$ in $D([0,T],\mathcal{S}'(\mathbb R))$ along some subsequence $n$. 
Moreover, the limit $\mathcal{X}^\sigma$ clearly satisfies the condition \textbf{(S)}. 
Therefore, we can show the following convergence. 
\begin{equation*}
\mathcal{A}^{\sigma,\varepsilon}_{s,t}(\varphi)
= \lim_{n\to\infty} 
\frac{1}{n}\int_0^t \sum_{j\in\mathbb Z} 
\mathcal{X}^{\sigma,n}_r \big( \iota_\varepsilon
({\textstyle \frac{j-v^\sigma_nr}{n}};\cdot) \big)^2 
\partial_x \varphi(\tfrac{j}{n}+ v^\sigma_nr) dr 
\end{equation*}
where $\mathcal{A}^{\sigma,\varepsilon}_{s,t}(\varphi)$ coincides with the process we defined in \eqref{eq:def_quadratic_function_approximation} with $u = \mathcal{X}^\sigma$, by approximating the function $\iota_\varepsilon$ by functions in $\mathcal{S}(\mathbb R)$, see \cite[Section 5.3]{gonccalves2014nonlinear} for this argument).  
On the other hand, by the second-order Boltzmann-Gibbs principle~(\cref{prop:2bg}), 
\begin{equation*}
\begin{aligned}
&\mathbb E_n \Bigg[ \bigg| \mathcal{B}^{\sigma,n}_t(\varphi) 
- \mathcal{B}^{\sigma,n}_s(\varphi) 
+ \sigma\alpha \frac{c_3}{8c_2^2} 
\int_s^t \sum_{j\in\mathbb Z} \big( \overrightarrow{(\xi^\sigma)}^{\lfloor \varepsilon n\rfloor }_j (r) \big)^2 
\partial_x \varphi(\tfrac{j}{n}+v^\sigma_nr)  dr 
\bigg|^2 \Bigg] \\
&\quad\lesssim 
\Big( \frac{(t-s)\varepsilon n}{n} 
+ {\frac{t-s}{\varepsilon n}} \Big) \| \partial_x \varphi \|^2_{L^2(\mathbb R)} .
\end{aligned}
\end{equation*}
Let $n\to \infty$ to obtain 
\begin{equation}
\label{eq:energy_condition_derivation}
\mathbb E \bigg[\Big| \mathcal{B}^{\sigma}_t(\varphi) 
- \mathcal{B}^{\sigma}_s(\varphi) 
+\frac{c_3\alpha}{8} \mathcal{A}^{\sigma,\varepsilon}_{s,t}(\varphi) \Big|^2 \bigg] 
\lesssim \varepsilon (t-s) \| \partial_x \varphi\|^2_{L^2(\mathbb R)} .
\end{equation}
Hence, the triangle inequality yields the condition \textbf{(EC)}. 
Consequently, by Proposition \ref{prop:nonlinear}, we get the existence of the limit 
\begin{equation*}
\mathcal{A}^\sigma_{t}(\varphi)
= \lim_{\varepsilon \to 0} \mathcal{A}^{\sigma, \varepsilon}_{0,t}(\varphi).  
\end{equation*}
Moreover, setting $s=0$ in the estimate \eqref{eq:energy_condition_derivation}, we find 
\begin{equation*}
\mathcal{B}^{\sigma}= - \sigma\alpha \frac{c_3}{8c_2^2}  \mathcal{A}^\sigma. 
\end{equation*}
Hence, up to now, we conclude that the condition (2) in \cref{def:energysol} holds.

Finally, we give some comments to complete the proof of Theorem \ref{thm:sbe_derivation_from_chain}. 
Note that all the above estimates hold \emph{mutatis mutandis} for the reversed process $\{ \mathcal{X}^{\sigma,n}_{T-t}: t\in [0,T] \}$ by repeating the argument for the dynamics generated by $L^*_n$, from which we conclude that the condition (3) of Definition \ref{def:energysol} is satisfied. 
Therefore, it follows that that the limiting process $u^\sigma$ is the stationary energy solution of the SBE \eqref{eq:drift-sbe}. 
Moreover, the independence of the two white-noises for $u^+$ and $u^-$ follows from the computation in \cref{sec:riemann_lebesgue_estimates}.  
Hence we complete the proof of \cref{thm:sbe_derivation_from_chain}.

\appendix

\section{Stationary Energy Solution of the
Stochastic Burgers Equation}
\label{sec-ssbe}

Let us recall the notion of the stationary energy solution of the stochastic Burgers equation which was introduced in~\cite{gonccalves2014nonlinear}. 
Let $\beta,D > 0$ and $\Lambda \in \mathbb{R}$ be fixed constants and consider the $(1+1)$-dimensional stochastic Burgers equation
\begin{equation}
\label{eq:SBE_general} 
\partial_t u = D \partial_x^2 u 
+ \Lambda \partial_x u^2 + \sqrt{D\beta^{-1}} \partial_x \dot{W}.
\end{equation}
We begin with the definition of stationarity. 

\begin{definition}
We say that an $\mathcal{S}^\prime (\mathbb{R})$-valued process $u = \{ u_t  : t \in [0,T] \} $
satisfies condition \textbf{(S)} if for all $t \in [0,T]$, the random variable $u_t$
has the same distribution as the space white-noise with variance $\beta^{-1}$. 
\end{definition}

For a process $u = \{ u_t: t \in [0,T]\}$ satisfying the condition \textbf{(S)}, we define 
\begin{equation}
\label{eq:def_quadratic_function_approximation}
\mathcal{A}^\varepsilon_{ s, t } (\varphi ) = \int_s^t \int_{\mathbb{R} } u_r (\iota_\varepsilon (x; \cdot) )^2 \partial_x \varphi (x ) dx dr .  
\end{equation}
for every $0 \le s < t \le T $, $\varphi \in \mathcal{S} (\mathbb{R} ) $ and $\varepsilon > 0 $. 
Here we defined the function $\iota_\varepsilon (x ; \cdot ) : \mathbb{R} \to \mathbb{R}  $ by $\iota_{ \varepsilon } (x ; y) =  \varepsilon^{ - 1 } \mathbf{1}_{ [ x , x  + \varepsilon ) } (y) $ for each $x \in \mathbb{R} $ and $\epsilon>0$. 
Although the function $\iota_\varepsilon(x,\cdot)$ does not belong to the Schwartz space, the quantity \eqref{eq:def_quadratic_function_approximation} is well-defined when $u$ satisfies the condition \textbf{(S)}.

\begin{definition}
Let $u = \{ u_t :t \in [0,T]\}$ be a process satisfying the condition \textbf{(S)}. 
We say that the process $u$ satisfies the energy estimate \textbf{(EC)} if there exists a constant $\kappa > 0$ such that for any $\varphi \in \mathcal{S} (\mathbb{R} )$, any $0 \le s < t \le T$ and any $0 < \delta < \varepsilon < 1 $,  
\begin{equation*}
\mathbb{E} \Big[ \big| \mathcal{A}^\varepsilon_{ s, t } (\varphi )
- \mathcal{A}^\delta_{ s, t } (\varphi ) \big|^2 \Big] 
\le \kappa \varepsilon (t-s) \| \partial_x \varphi \|^2_{ L^2(\mathbb{R})} .
\end{equation*}
Here $\mathbb{E}$ denotes the expectation with respect to the measure $\mathbb P$ of a probability space where the process $u$ lives. 
\end{definition}

Then the following result is proved in \cite{gonccalves2014nonlinear}. 

\begin{proposition}
\label{prop:nonlinear}
Assume $\{ u_t:t\in [0,T]\} $ satisfies the conditions \textbf{(S)} and \textbf{(EC)}. Then there exists an $\mathcal{S}^\prime (\mathbb{R} )$-valued process $\{ \mathcal{A}_t : t \in [0, T ] \} $ with continuous trajectories such that  
\begin{equation*}
\mathcal{A}_t (\varphi ) = \lim_{ \varepsilon \to 0 } \mathcal{A}^\varepsilon_{ 0, t } (\varphi) ,
\end{equation*}
in $L^2(\mathbb P)$ for every $t \in [0,T]$ and $\varphi \in \mathcal{S}(\mathbb{R})$.  
\end{proposition}

From the last proposition, thinking that the singular term $\partial_x u^2 $ is given by the last quantity, we can define a solution of \eqref{eq:SBE_general} as follows.

\begin{definition}
\label{def:energysol}
We say that an $\mathcal{S}^\prime(\mathbb{R})$-valued process $u=\{u (t, \cdot) : t\in [0,T] \}$ is a stationary energy solution of the stochastic Burgers equation \eqref{eq:SBE_general} if 
\begin{enumerate}
\item The process $u$ satisfies the conditions \textbf{(S)} and \textbf{(EC)}. 
\item For all $\varphi \in \mathcal{S} (\mathbb{R} )$, the process 
\begin{equation*}
u_t(\varphi) - u_0 (\varphi) - \nu \int_0^t u_s (\partial_x^2 \varphi ) ds 
+ \Lambda \mathcal{A}_t (\varphi) ,
\end{equation*}
is a martingale with quadratic variation $D \| \partial_x \varphi \|^2_{ L^2 (\mathbb{R} ) } t $ where $\mathcal{A}_\cdot$ is the process obtained in Proposition \ref{prop:nonlinear}. 
\item For all $\varphi \in \mathcal{S} (\mathbb{R} )$, writing $\hat{u}_t = u_{T-t}$ and $\hat{ \mathcal{A} }_t = - (\mathcal{A}_T - \mathcal{A}_{ T- t })$, the process
\begin{equation*}
\hat{u}_t (\varphi) - \hat{u}_0 (\varphi ) - \nu \int_0^t \hat{u}_s (\partial_x^2 \varphi) ds 
+ \Lambda \hat{\mathcal{A}}_t (\varphi) ,
\end{equation*}
is a martingale with quadratic variation $D \| \partial_x \varphi \|^2_{L^2(\mathbb{R})}t$. 
\end{enumerate}
\end{definition}

Then there exists a unique-in-law stationary energy solution of \eqref{eq:SBE_general}. 
Existence was shown in \cite{gonccalves2014nonlinear} and then uniqueness was proved in \cite{gubinelli2018energy}.

\section{Proof of \texorpdfstring{\cref{lem:static_estimate}}{} }
\label{sec:static_estimate}
Here, let us give a proof of \cref{lem:static_estimate} that we postponed in the main part. 
First, note that it is enough to prove the assertion for $\tau = 0$. 
We need to find two functions $W_\pm(r)$ such that $\int e^{-\beta W_+(r)} dr <+\infty$ and
$\int e^{-\beta W_-(r) + \eta |r|} dr <+\infty$ for any $\eta$, and for any $n>0$ the following is satisfied:
\begin{equation}
\label{eq:3}
W_-(r) \le V_n (r) \le  W_+(r).
\end{equation}
Then, it is easy to check that \eqref{eq:uniform_moment_bound} is satisfied with the choice   
$$
C_\eta = \frac{\int e^{-\beta W_-(r) + \eta |r|} dr }{\int e^{-\beta W_+(r)} dr}. 
$$ 
{A typical example of} potential satisfying Assumption \ref{asm:potential} is the Toda lattice interaction:
$$
V^{\mathrm{Toda}}_{\bar\eta}(r) = e^{-\bar\eta r} + \bar\eta r -1,
$$
which has linear growth for $r>0$ and exponential growth for $r<0$.
Here note that $V_n(r) \to \frac{\bar\eta^2}{2} r^2$ as $n \to \infty$. 
This suggests to choose $W_+(r) = e^{\bar\eta |r|}$ for $\bar\eta < \eta_V$, then
$V_n (r) \le  W_+(r)$ for $n$ large enough.
For the lower function, it is not hard to see that a linear function $W_-(r) = a |r| - b$ with properly chosen $a,b>0$
satisfies the desired bound.

\section{Auxiliary Estimates}
In the following let $v\in \mathbb R$ and $\varphi\in L^2(\mathbb R)$ and we use the short-hand notation
$\varphi(j,t)= \varphi(\tfrac{j}{n} + vt)$ \Add{and assume that we have $\sum_{j\in\mathbb Z} \varphi(j,t)=0$, which is indeed the case since we apply the forthcoming estimates only for test functions with discrete derivative.}

\begin{lemma}
\label{lem:gen_bound}
We have that 
\begin{equation*}
\begin{split}
& \mathbb E_n \bigg[\sup_{0\le t\le T} \bigg| \sum_{j\in\mathbb Z}  
p_j(t) r_{j+1}(t)^{2} 
\varphi(j,t) \bigg|^2\bigg]   
\lesssim 
n^3 
\left( \| \partial_x \varphi \|^2_{L^2(\mathbb R)} + \| \varphi \|^2_{L^2(\mathbb R)}\right) . 
\end{split}
\end{equation*}
\end{lemma}
\begin{proof} 
\Add{
Using an elementary inequality $ab\le (a^2 + b^2)/2$, it is enough to bound 
\begin{equation*}
\mathbb E_n\Big[\sup_{0\le t\le T} 
\big| \mathcal P^n_t(\varphi)\big|^2 \Big]
\quad \text{and} \quad
\mathbb E_n\Big[\sup_{0\le t\le T} 
\big| \mathcal K^n_t(\varphi)\big|^2 \Big]
\end{equation*}
where we denoted for brevity
\begin{equation*}
\mathcal{P}^n_t(\varphi)
= \sum_{j\in\mathbb Z} r_{j}(t)^{4} \varphi(j,t),
\quad\text{and} \quad
\mathcal{K}^n_t(\varphi)
= \sum_{j\in\mathbb Z} p_{j}(t)^{2} \varphi(j,t).  
\end{equation*}
Let us firstly bound the (kinetic) energy fluctuation. 
By applying Dynkin's formula, 
\begin{equation*}
\mathfrak M_t(\varphi)
=  \mathcal{K}^n_t(\varphi)
- \mathcal{K}^n_0(\varphi)
- \int_0^t (n^2 L + \partial_s) \mathcal{K}^n_s(\varphi)ds 
\end{equation*}
is a martingale whose quadratic variation is given by 
\begin{equation*}
\begin{aligned}
\langle \mathfrak M(\varphi)\rangle_t
&= n^2 \int_0^t \big(L\mathcal K^n_s(\varphi)^2 -2\mathcal{K}^n_s(\varphi)L\mathcal{K}^n_s(\varphi) \big)ds\\
&= \frac{\gamma}{4} \int_0^t \sum_{j\in\mathbb Z} (p_j(s)^2 - p_{j+1}(s)^2)^2 
\big(\nabla^n\varphi(j, s)\big)^2 ds. 
\end{aligned}
\end{equation*}
Consequently, by Doob's inequality, we have that 
\begin{equation*}
\begin{aligned}
\mathbb E_n\Big[\sup_{0\le t\le T} \mathfrak M_t(\varphi)^2 \Big] 
\le 4\mathbb E_n\big[\mathfrak M_T(\varphi)^2 \big]
\lesssim T n \| \partial_x \varphi \|^2_{L^2(\mathbb R)}. 
\end{aligned}
\end{equation*}
On the other hand, noting $Sp_j^2=\gamma \Delta p_j^2$, we have 
by \eqref{eq:h-1}
\begin{equation*}
\begin{split}
\mathbb E_n \bigg[ \sup_{0\le t\le T} \bigg| \int_0^t n^2 S  \mathcal{K}^n_s(\varphi) ds \bigg|^2 \bigg]
& \le 24 n^2 \int_0^T \langle \mathcal{K}^n_s(\varphi) ,-S \mathcal{K}^n_s(\varphi) \rangle_{L^2(\nu_n)}ds \\
&\le 24 n^2 \gamma \int_0^T \sum_{j,j'}
\mathbb E_n[ p_j^2 \Delta p_{j'}^2]
\varphi(j,s)\varphi(j',s) ds \\
&\lesssim n^2 T \sum_{j}\varphi_{j}^2 
\lesssim n^3 \|\varphi\|^2_{L^2(\mathbb R)}.
\end{split}
\end{equation*}
Moreover, note that 
\begin{equation*}
\begin{aligned}
Ap_j^2 
=2\alpha(V'_n(r_{j+1}) - V'_n(r_j))p_j. \end{aligned}
\end{equation*}
Now, decompose 
\begin{equation}
\label{eq:anti_symmetric_part_computation_kinetic}
\begin{aligned}
\sum_{j\in\mathbb Z}
Ap_j^2\varphi(j,t)
&= 2\gamma\sum_{j\in\mathbb Z} 
V_n'(r_j)
p_{j-1}(\varphi(j-1,t)-\varphi(j,t)) \\
&\quad+ 2\gamma\sum_{j\in\mathbb Z} 
V_n'(r_j)
(p_{j-1}-p_j)\varphi(j,t). 
\end{aligned}
\end{equation}
Then, we notice that the first term in the right-hand side of \eqref{eq:anti_symmetric_part_computation_kinetic}, which includes the derivative in space, absorbs $n$.
Meanwhile, for the second term, we apply \eqref{eq:h-1} to see 
\begin{equation*}
\begin{aligned}
&\mathbb E_n\bigg[\sup_{0\le t\le T} 
\bigg|\int_0^t \sum_{j\in\mathbb Z} 
n^2V_n'(r_j(s))(p_j(s)-p_{j-1}(s)) \varphi(j,s) ds\bigg]\\
&\quad\le 24n^2\int_0^T 
\bigg\| 
\sum_{j\in\mathbb Z} 
(p_j-p_{j-1})V_n'(r_j)\varphi(j,s) \bigg\|^2_{-1,n}ds 
\lesssim n^2 \| \varphi\|^2_{L^2(\mathbb R)}. 
\end{aligned}
\end{equation*}
Thus, we have that 
\begin{equation*}
\begin{aligned}
\mathbb E_n \bigg[ \sup_{0\le t\le T} \bigg| \int_0^t n^2 A  \mathcal{K}^n_s(\varphi) ds \bigg|^2 \bigg]
\lesssim n^3\big(\|\varphi\|^2_{L^2(\mathbb R)}
+ \|\partial_x \varphi\|^2_{L^2(\mathbb R)}\big). 
\end{aligned}
\end{equation*}
In addition, note that 
\begin{equation*}
\begin{aligned}
\mathbb E_n \bigg[\sup_{0\le t\le T} \bigg| \int_0^t \partial_s \mathcal{K}^n_s(\varphi) ds \bigg|^2 \bigg]
\lesssim T^2 n^3v^2 \| \partial_x \varphi\|^2_{L^2(\mathbb R)}. 
\end{aligned}    
\end{equation*}
and note further that a direct estimate of $\mathcal K^n_0(\varphi)$ gives the bound $n\|\varphi\|^2_{L^2(\mathbb R)}$. 
By this line, we get the desired bound for the fluctuation field $\mathcal K^n_\cdot$. 
}

\Add{
Next, let us bound the fluctuation field $\mathcal P^n_\cdot$, which has the following time evolution: 
\begin{equation*}
\begin{aligned}
\mathcal P^n_t(\varphi)
= \mathcal P^n_0(\varphi)
+ \int_0^t (n^2A+ \partial_s) \mathcal P^n_s(\varphi)ds 
\end{aligned}
\end{equation*}
where we used the fact that $Sr_j^4=0$. 
However, since $Ar_j^{4}
= 4 (p_{j}-p_{j-1}) r_j^{3}$, for $\int_0^tn^2A\mathcal P^n_s(\varphi)ds$, we can apply an analogous argument as in the second term of the right-hand side of~\eqref{eq:anti_symmetric_part_computation_kinetic}, and the other terms can be handled easily.
Hence we complete the proof by combining all the above estimates. 
}

\end{proof}

\Add{
\begin{lemma}
\label{lem:consequence_one_block_estimate}
Let $F:\mathscr X\to\mathbb R$ be a local $L^2(\nu_n)$-function of the form $F(\mathfrak r,\mathfrak p)=F_0(r_0,r_1,\ldots,r_{k_r},p_1,\ldots,p_{k_p})$ with some $k_r,k_p\in\mathbb Z_+$ and $F\in L^2(\nu_n)$, and assume that it satisfies $E_{\nu_n}[F]=0$. 
Moreover, let $G_i:\mathbb R\to\mathbb R$
such that $E_{\nu_n}[G_i(p_j)]=0$ for each $i=1,\ldots,m$. 
Then, we have that 
\begin{equation*}
\begin{aligned}
\mathbb E_n\bigg[\sup_{0\le t\le T}\bigg| \int_0^t 
\sum_{j\in\mathbb Z}  
\prod_{i=1}^mG_i(p_{j+i}(s)) \tau_j F(\mathfrak r(s))
\varphi(j,s)  ds \bigg|^2 \bigg]
\lesssim 
T^{3/2} \| \varphi\|^2_{L^2(\mathbb R)}.
\end{aligned}
\end{equation*}
\end{lemma}
\begin{proof}
First, note that we may assume $m=1$, since it turns out that the general assertion follows by regarding $\prod_{1\le i \le m-1}G_i(p_{i+j}) \tau_j F(\mathfrak r)$ as the function $F$ as repeat the same argument.  
We decompose 
\begin{equation*}
\begin{aligned}
G(p_j) \tau_j F(\mathfrak r)
= (G(p_j)-\overrightarrow{G(p_j)}^{\delta n}) \tau_j F(\mathfrak r)
+ \overrightarrow{G(p_j)}^{\delta n} \tau_j F(\mathfrak r). 
\end{aligned}
\end{equation*}
for any $\delta>0$ where we used the same notation for the integer part of $\delta n$ by abuse of notation.  
Then, analogously to the one-block estimate (Lemma \ref{lem:one_block_estimate}), we have that 
\begin{equation*}
\mathbb E_n \bigg[\sup_{0\le t \le T} \bigg| \int_0^t
\sum_{j\in\mathbb Z}
(G(p_j(s)) -\overrightarrow{G(p_j)}^{\delta n}(s)) \tau_j F(\mathfrak r(s)) 
\varphi(j,s)   ds \bigg|^2 \bigg]
\lesssim T\delta
\|\varphi\|^2_{L^2(\mathbb R)}. 
\end{equation*}
On the other hand, by Cauchy-Schwarz inequality, we have the bound 
\begin{equation}
\label{eq:cs}
\begin{aligned}
\mathbb E_n \bigg[ \sup_{0\le t \le T} \bigg| \int_0^t
\sum_{j\in\mathbb Z}
\overrightarrow{G(p_j)}^{\delta n}(s) \tau_j F(\mathfrak r(s)) 
\varphi(j,s)   ds \bigg|^2 \bigg] 
&\le C T^2 \sum_{j\in\mathbb Z}\mathbb E_n \big[
(\overrightarrow{G(p_j)}^{\delta n})^2 \big]
\varphi(\tfrac jn)^2 \\
&\le\frac{C T^2}{\delta} \|\varphi \|^2_{L^2(\mathbb R)}
\end{aligned}
\end{equation}
where we used the fact that $E_{\nu_n}[(\overrightarrow{G(p_j)}^{\delta n})^2]\lesssim 1/(\delta n)$. 
Note that the above two estimates are optimized when $\delta=\sqrt{T}$. 
Thus, we obtain the desired bound. 
\end{proof}
}

Similarly we have the following Lemma.
\begin{lemma}
\label{lem:rr}
We have that 
\begin{equation*}
\begin{aligned}
\mathbb E_n\bigg[\sup_{0\le t\le T}\bigg| \int_0^t 
\sum_{j\in\mathbb Z} V'_n(r_{j+1}(s)) \overline{r_{j}^2}
\varphi(j,s)  ds \bigg|^2 \bigg]
\lesssim 
T^{3/2} \| \varphi\|^2_{L^2(\mathbb R)}. 
\end{aligned}
\end{equation*}
\end{lemma}

\begin{proof}
We proceed as in the proof of Lemma
\ref{lem:consequence_one_block_estimate}.
Denote by $V'_j =  V'_n(r_{j})$. For any small $\delta >0$
we decompose
\begin{equation*}
\begin{aligned}
V'_{j+1} \overline{r_{j}^2}
= (V'_{j+1}-\overrightarrow{ (V')_{j+1}}^{\delta n}) \overline{r_{j}^2}
+ \overrightarrow {(V')_{j+1}}^{\delta n} \overline{r_{j}^2}.
\end{aligned}
\end{equation*}
As in \eqref{eq:cs} we have
\begin{equation*}
\mathbb E_n \bigg[ \sup_{0\le t \le T} \bigg| \int_0^t
\sum_{j\in\mathbb Z}
\overrightarrow {(V')_{j+1}}^{\delta n}(s) \overline{r_{j}^2}(s)
\varphi(j,s)  ds \bigg|^2 \bigg]
\le \frac{T^2}{\delta} \|\varphi \|^2_{L^2(\mathbb R)}.
\end{equation*}
On the other hand, let us write 
\begin{equation*}
\begin{split}
(V'_{j+1}-\overrightarrow{ (V')_{j+1}}^{\delta n}) \overline{r_{j}^2}
&= \sum_j \overline{r_{j}^2} \sum_{i=1}^{\delta n -1} (V'_{j+i} - V'_{j+i+1}) \psi_{i-1}\\
&= \sum_j \overline{r_{j}^2} \sum_{i=1}^{\delta n-1} A p_{j+i} \psi_{i-1}\\
&=  L \sum_j \overline{r_{j}^2} \sum_{i=1}^{\delta n-1} p_{j+i} \psi_{i-1}
- 2 \sum_j r_{j} (p_j - p_{j-1}) \sum_{i=1}^{\delta n -1} p_{j+i} \psi_{i-1} \\
&\quad- \gamma \sum_j \overline{r_{j}^2} \sum_{i=1}^{\delta n-1} (\Delta p_{j+i}) \psi_{i-1}.
\end{split}
\end{equation*} 
Here, note that each term of the utmost right-hand side of the last display can be analyzed using the noise on the velocities. 
Indeed, since the first term is in the domain of the generator, it turns out to be negligible following an analogous argument as in \cref{lem:equipartition_energy}, given the fact that boundary terms in a martingale decomposition is suppressed by \cref{lem:gen_bound}. 
On the other hand, the other two terms, up to constant, are bounded by $T\delta \|\varphi\|^2_{L^2(\mathbb R)}$ by applying \cref{prop:kipnis_varadhan_estiamte}. 
Hence, optimizing in $\delta$, we obtain the desired bound. 
\end{proof}

\Add{
\begin{lemma}
\label{lem:rr2}
We have that 
\begin{equation*}
\begin{aligned}
\mathbb E_n\bigg[\sup_{0\le t\le T}\bigg| \int_0^t 
\sum_{j\in\mathbb Z} \overline r_{j+1}(s) \overline{r_{j}^2}(s) 
\varphi(j,s)  ds \bigg|^2 \bigg]
\lesssim 
T^{3/2} \| \varphi\|^2_{L^2(\mathbb R)}. 
\end{aligned}
\end{equation*} 
\end{lemma}
\begin{proof}
Since $V'(r_j)=c_2\overline r_j + \frac{c_3}{2}\varepsilon_n \overline r_j^2 + O(\varepsilon_n^2)$, by virtue of \cref{lem:rr} and a crude $L^2$-estimate, it suffices to see that 
\begin{equation*}
\int_0^t \sum_{j\in\mathbb Z} \varepsilon_n  \overline{r_j^2}(s) \overline{r_{j+1}^2}(s) 
\varphi(j,s)ds  
\end{equation*}
is negligible. 
To this end, note that 
\begin{equation*}
\begin{aligned}
A(r_j^2r_{j+1}p_{j+1})
&= r_j^2 r_{j+1}(V_n'(r_{j+2})-V_n'(r_{j+1})) \\
&\quad+ 2r_jr_{j+1}(p_j-p_{j-1})p_{j+1}
+ r_j^2 (p_{j+1}-p_j) p_{j+1} \\
&= c_2 r_j^2 r_{j+1}(r_{j+2}-r_{j+1})  \\
&\quad+ 2r_jr_{j+1}(p_j-p_{j-1})p_{j+1}
+ r_j^2 (p_{j+1}-p_j) p_{j+1} 
+O(\varepsilon_n)
\end{aligned}
\end{equation*}
and this whole quantity, which is in the range of $A$, does not survive in the limit. 
Here, $E_{\nu_n}[r_j]=O(\varepsilon_n)$ enables us to recenter variables without the square, whereas the recentering of $r_j^2$ and $r_{j+1}^2$ is justified by the equipartition of energy (\cref{lem:equipartition_energy}).
Therefore, using Lemma \ref{lem:consequence_one_block_estimate}, the identity in the last display shows that 
$$
\lim_{n\to\infty}
\mathbb E_n\left[
\sup_{0\le t\le T} 
\bigg|\varepsilon_n\int_0^t\sum_{j\in\mathbb Z} 
\Big(\overline{r_j^2} \left(c_2 \overline{r_{j+1}}\; \overline{r_{j+2}}- p_j p_{j+1}\right) +c_2\overline{r_j^2} \; \overline{r_{j+1}^2} - \overline{r_j^2} \; \overline{p_{j+1}^2}\Big) \varphi(j,s) ds \bigg|^2 \right] = 0.
$$
By Lemma~\ref{lem:consequence_one_block_estimate}, we can apply the block estimate to $p_jp_{j+1}$, and $\overline{r_j^2}\; \overline{p_{j+1}^2}$ to see that those terms do not contribute in the limit in the sense of the last display.  
Therefore, we are left to prove that the fluctuation of $c_2\overline{r_j^2} \overline{r_{j+1}}\; \overline{r_{j+2}}$ is also negligible
To this end, note that 
\begin{equation*}
c_2\overline{r_j^2} \overline{r_{j+1}}\; \overline{r_{j+2}}
- c_2\overline{r_j^2} \overline{r_{j+1}} V_n'(r_{j+2}) 
= O(\varepsilon_n), 
\end{equation*}
and the right-hand side of the last display is negligible by a crude $L^2$-estimate. 
Now, it is not hard to see that an analogous argument as in Lemma~\ref{lem:rr} enables us to deduce 
\begin{equation*}
\lim_{n\to\infty}
\mathbb E_n\left[
\sup_{0\le t\le T} 
\bigg|\varepsilon_n\int_0^t\sum_{j\in\mathbb Z} 
\overline{r_j^2} \overline{r_{j+1}} V_n'(r_{j+2}) 
\varphi(j,s) ds \bigg|^2 \right] 
= 0
\end{equation*}
and thus we conclude the proof.   
\end{proof}
}

\section{Proof of Lemma \ref{lem:corr}}
\label{sec:proof-lemma-refl}
{
First notice that we have the following expression: 
\begin{equation}
\label{eq:approx2}
\begin{split}
\mathcal{X}^{\sigma,n}_t(\varphi_{v^\sigma_nt})
- \widetilde{\mathcal{X}}^{\sigma,n}_t(\varphi_{v^\sigma_nt}) 
= \frac{\mathfrak u}{n} \sum_{j\in\mathbb Z} \overline{e_j}(t)  \varphi  \big(\tfrac{j}{n}+ {v^\sigma_nt}\big)
= \frac{\mathfrak u}{\sqrt n}{\mathcal{X}}^{0,n}_t(\varphi_{v^\sigma_nt})  .
\end{split}
\end{equation}
We can easily see that for any $t\ge 0$, it holds that 
\begin{equation}\label{eq:l2conv}
\begin{split}
\mathbb E_n\bigg[
\Big| \mathcal{X}^{\sigma,n}_t(\varphi_{v^\sigma_nt})
- \widetilde{\mathcal{X}}^{\sigma,n}_t(\varphi_{v^\sigma_nt})\Big|^2 \bigg]\le
\frac{3\mathfrak u^2}{n^2\beta^{-2}} \sum_{j\in\mathbb Z}  \varphi \big(\tfrac{j}{n}\big)^2
\le \frac{C}{n}\|\varphi\|_{L^2}^2.
\end{split}
\end{equation}
To finalize the result we need to prove that for $\sigma = \pm 1$
\begin{equation}
\label{eq:nullenergy}
\limsup_{n\to\infty}
\frac 1n \mathbb E_n\bigg[\sup_{0\le t \le T}
\Big| \mathcal{X}^{0,n}_t(\varphi_{v^\sigma_nt}) \Big|^2 \bigg]
\le C.
\end{equation}
This implies that the distributions of the sequence
$\{\frac 1{\sqrt n} \mathcal{X}^{0,n}_t(\varphi_{v^\sigma_nt})\}$
are tight in $\mathcal M_1(D([0,T],\mathbb R))$. 
This, together with \eqref{eq:l2conv}, implies that
$\mathcal{X}^{\sigma,n}_t(\varphi_{v^\sigma_nt}) -
\widetilde{\mathcal{X}}^{\sigma,n}_t(\varphi_{v^\sigma_nt})$ converges to $0$ on the space
$D([0,T],\mathbb R) $ in law.

The bound \eqref{eq:nullenergy} follows from the following argument: by the evolution equations we have
\begin{equation}
\label{eq:ene-evol}
\begin{split}
\frac 1{\sqrt n}
\left( \mathcal{X}^{0,n}_t(\varphi_{v^\sigma_nt}) - \mathcal{X}^{0,n}_0(\varphi)\right)
=  \int_0^t  \sum_j p_j(s) V'_n(r_j(s)) \varphi'_{v^\sigma_nt}(\tfrac jn) ds \\
+ \sigma \alpha \sqrt{c_2}  \int_0^t \sum_j \overline{e_j}(s)
\varphi'_{v^\sigma_nt}(\tfrac jn) ds .
\end{split}
\end{equation}
Then, it is enough to show the following bounds: 
\begin{equation}
\label{eq:ev-bounds1}
\begin{split}
\mathbb E_n\bigg[\sup_{0\le t \le T}
\Big|\int_0^t \sum_j p_j(s) V'_n(r_j(s))
\varphi'_{v^\sigma_nt}(\tfrac jn) ds \Big|^2 \bigg] \le C,
\end{split}
\end{equation}
\begin{equation}
\label{eq:ev-bounds2}
\begin{split}
\mathbb E_n\bigg[\sup_{0\le t \le T}
\Big|\int_0^t  \sum_j \overline{e_j}(s) \varphi'_{v^\sigma_nt}(\tfrac jn)ds 
\Big|^2 \bigg] \le C.
\end{split}
\end{equation}
The first bound \eqref{eq:ev-bounds1} can be shown analogously to \cref{lem:consequence_one_block_estimate}.
For, the second \eqref{eq:ev-bounds2} bound, note that the equipartition of energy (\cref{lem:equipartition_energy}) enables us to write main terms of the energy element $e_j$ using degree-two terms, and thus the same argument as in the first bound \eqref{eq:ev-bounds1} works, where we omit the details to avoid redundancy.

\section*{Acknowledgments}
K.H. was supported by JSPS KAKENHI Grant Number 22J12607 and 25K23337.
S.O. was supported by the Institute Universitaire de France.
S.O. appreciates the hospitality of The University of Tokyo where part of this work was accomplished.
Additionally, the authors would like to thank Makiko Sasada for giving them fruitful comments and suggestions.

\bibliographystyle{abbrv}
\bibliography{ref-s}

\end{document}